\documentclass{amsart}
\usepackage{amsmath,amsthm}
\usepackage{amsfonts,amssymb}

\usepackage{enumerate}
\usepackage{multirow}

\newtheorem{thm}{Theorem}[section]

\newtheorem{lem}[thm]{Lemma}
\newtheorem{prop}[thm]{Proposition}

\theoremstyle{remark}
\newtheorem{rem}{Remark}[section]

\def\f{\frac}

 \def\a{{\alpha}} 
 \def\b{{\beta}}
 \def\g{{\gamma}}
 
 \def\t{{\theta}}
 \def\l{{\lambda}}

 \def\la{{\langle}}
 \def\ra{{\rangle}}

 \def\xb{{\mathbf x}}

 \def\C{{\mathcal C}}

 \def\CI{{\mathcal I}}
 
 \def\CK{{\mathcal K}}
 \def\CL{{\mathcal L}}
 \def\CO{{\mathcal O}}
 \def\CP{{\mathcal P}}

 \def\CV{{\mathcal V}}
 \def\CW{{\mathcal W}}

 \def\NN{{\mathbb N}}
 \def\PP{{\mathbb P}}
 \def\QQ{{\mathbb Q}}
 \def\RR{{\mathbb R}}

\newcommand{\wh}{\widehat}

\iffalse
\newif\ifpdf
\ifx\pdfoutput\undefined
  \pdffalse
\else
  \pdftrue
\fi

\ifpdf
  \usepackage[pdftex]{graphicx}
  \DeclareGraphicsExtensions{.pdf,.jpg,.png}
\else
  \usepackage{graphicx}
\fi

\begin{document}
 
\title[Minimal cubature rules and interpolation]
{Minimal Cubature rules and polynomial interpolation 
in two variables II}
\author{Yuan Xu} 
\address{Department of Mathematics\\ University of Oregon\\
    Eugene, Oregon 97403-1222.}\email{yuan@math.uoregon.edu}

\date{\today}
\thanks{The work was supported in part by NSF Grant DMS-1510296}
 
\begin{abstract}
As a complement to \cite{X12}, minimal cubature rules of degree $4m+1$ for the weight functions 
$$
\CW_{\a,\b,\pm \frac12}(x,y) = |x+y|^{2\a+1} |x-y|^{2\b+1} ((1-x^2)(1-y^2))^{\pm \frac12}
$$
on $[-1,1]^2$ are shown to exist and near minimal cubature rules of the same degree with one node more
than minimal are constructed explicitly. The Lagrange interpolation polynomials on the nodes of the 
near minimal cubature rules are also studied. 
\end{abstract}

\maketitle

\section{Introduction}
\setcounter{equation}{0}

Let $W$ be a non-negative weight function on a domain $\Omega \in \RR^2$. A cubature formula of degree 
$s$ for the integral with respect to $W$ is a finite sum satisfying 
\begin{equation} \label{eq:CF}
   \int_\Omega f(x,y) W(x,y) = \sum_{k=1}^N \l_k f(x_k,y_k), \qquad \forall f \in \Pi_s^2,
\end{equation}
where $\Pi_s^2$ denotes the space of polynomials of degree at most $s$ in two variables, and there exists 
at least one function $f^* \in \Pi_{s+1}^2$ for which the identity fails to hold. 
For a fixed $n$, a minimal cubature rule of degree $2n+1$ has the smallest number, $N$, of nodes among 
all cubature rules of the same degree. They are of interests in several aspects and provide important
tools for various problems in approximation and numerical computation. 

Minimal cubature rules are known explicitly, with their nodes and weights given by 
close formulas, only in a few cases. One of them is for the family of integrals with respect to the weight functions 
\begin{equation} \label{CWab}
 \CW_{\a,\b,\pm \frac12}(x,y) : = |x+y|^{2 \a+1} | x- y|^{2 \beta +1}
          (1-x^2)^{\pm \frac12}(1-y^2)^{\pm \frac12} 
\end{equation}
with $\a,\b > -1$ on the square $[-1,1]^2$, for which the minimal cubature rules of degree $2n-1$ are 
explicitly constructed in \cite{X12} when $n = 2m$. This includes the classical result of the product Chebyshev
weight function (when $\a=\b = \-1/2$) studied in \cite{MP} as a special case. The cubature rules in \cite{X12}
are closely tied to Gaussian cubature rules on a domain bounded by two lines and a parabola. The nodes of 
these minimal cubature rules are common zeros of certain orthogonal polynomials with respect to 
$\CW_{\a,\b,\pm \f12}$, and, for each fixed $n$, there is a unique Lagrange interpolation polynomial based 
on the nodes of the minimal cubature rule. The case $n=2m+1$ was left open in \cite{X12} because the 
idea of the polynomials that vanish on its nodes has a more complicated structure that requires further
study to understand. 

The purpose of the present paper is to show how the case $n=2m+1$ can be resolved. We obtained two 
families of cubature rules in this case. The first family consists of minimal cubature rules, whose coefficients, 
however, are not explicitly given, whereas the second family consists of cubature rules whose number of 
nodes is 1 more than the theoretical lower bound, but it can be determined explicitly. The nodes of these 
cubature rules are common zeros of certain orthogonal polynomials of degree $n$ and, in the case of the 
second family, one quasi-orthogonal polynomial of degree $n+1$ that does not belong to the idea generated
by those orthogonal polynomials of degree $n$. The second family of cubature rules are explicitly constructed 
because they are related to the product Gauss-Radau cubature rules with respect to the product Jacobi 
weights. For all practical considerations, the second family is better and their study resembles the case 
of $n = 2m$ in \cite{X12}. In addition, we will also give explicit formulas of the Lagrange interpolation polynomials
based on the nodes of the near minimal cubature rules. These formulas allow us to determine the order 
of the Lebesgue constants of the interpolation operators.  

We regard this paper as a complement of \cite{X12} and will refer to the background materials and, in some
cases, even quote formulas there. However, we have tried to make the paper self-contained, so that it can 
be read independently. The paper is organized as follows. In the next section we state background materials,
highlight those not covered in \cite{X12}. The cubature formulas are studied in Section 3 and the Lagrange
interpolation polynomials based on the nodes of the cubature rules are discussed in Section 4. 

\section{Preliminary and Background}
\setcounter{equation}{0}

Besides the section on preliminary and background in \cite{X12}, we need background on near minimal 
cubature rules and orthogonal polynomials of odd degrees with respect to the weight function $\CW_{\a,\b,\pm \f12}$.  

\subsection{Near minimal cubature rules}
Let $W$ be a nonnegative weight function on a domain $\Omega$ in $\RR^2$ that has all finite moments. 
that is, $\int_\Omega x_1^j x_2^k W(x_1,x_2) dx_1dx_2 < \infty$ for all $j,k \in \NN_0$. 
Let 
$$
  \la f, g \ra_W := \int_\Omega f(x_1,x_2)g(x_1,x_2)W(x_1,x_2) dx_1dx_2.
$$
With respect to the weight function $W$, a polynomial $P\in \Pi_n^2$ is called an orthogonal polynomial if 
$\la P, Q\ra_W =0$ for all $Q\in \Pi_{n-1}^2$. Let $\CW_n(W)$ be the space of orthogonal polynomials of two 
variables. Then $\dim \CV_n(W) = n +1$. A basis of $\CV_n(W)$ can be denoted by $\PP_n:=\{P_{k,n}: 0 \le k \le n\}$; 
it is called a mutually orthogonal basis if $\la P_{k,n}, P_{j,n} \ra_W =0$ for $k \ne j$ and an orthonormal bias if, 
in addition, $\la P_{k,n}, P_{k,n}\ra_W =1$. Clearly $\PP_n$ is a set but it is often convenient to regard it also 
as a column vector.

A function $W$ is called centrally symmetric if $x \in \Omega$ implies $-x \in \Omega$ and $W(x) = W(-x)$. 
Evidently, $\CW_{\a,\b,-\pm \f12}$ is centrally symmetric. For a central symmetric integral, the number of nodes, 
$N$, of the cubature rule of degree $2n-1$ in \eqref{eq:CF} satisfies the lower bound \cite{M}
\begin{equation}\label{lwbd}
  N \ge \dim \Pi_{n-1}^2 + \left \lfloor \frac{n}{2} \right \rfloor = \frac{n(n+1)}{2} + \left \lfloor \frac{n}{2} \right \rfloor
     :=N_{\min} , 
\end{equation}
and it is known that the nodes of such a cubature rule are common zeros of a subspace of the space $\CV_n(W)$.
More precisely, the following theorem holds: 

\begin{thm} \label{thm:min}
A minimal cubature rule of degree $2n-1$ that attains the lower bound \eqref{lwbd} exists if 
and only if its nodes are common zeros of $\lfloor \frac{n+1}{2} \rfloor +1$ orthogonal polynomials of degree $n$.
\end{thm}

Let $X$ be the set of the nodes. In the language of idea and variety, the characterization states that $X$ is the
variety of the polynomial ideal generated by $\lfloor \frac{n+1}{2} \rfloor +1$ orthogonal polynomials of degree $n$.
This characterization, established in \cite{M}, was extended in \cite{X94} by considering $(2n-1)$-orthogonal 
polynomials. 

A polynomial $P$ is called a $(2n-1)$-orthogonal polynomial if $\int_\Omega P Q W dxdy =0$ for all polynomials 
$Q$ that satisfy $\deg P + \deg Q \le 2n-1$. Evidently, orthogonal polynomials themselves are $(2n-1)$-orthogonal
polynomials. Moreover, any polynomial of the form $P_{n+1} + P_n + P_{n-1}$, where $P_k \in \CV_k(W)$, is
$(2n-1)$-orthogonal, since they are orthogonal to all polynomials of degree $n-2$ and $(n+1)+(n-2) = 2n-1$. In 
\cite{X94}, we gave a characterization of minimal or near minimal cubature rules of degree $2n-1$ in terms of 
linearly independent $(2n-1)$-orthogonal polynomials that are components of 
\begin{equation} \label{quasiOP}
  \QQ_n:= \PP_{n+1}+ \Gamma_1 \PP_n + \Gamma_2 \PP_{n-1},
\end{equation}
where $\Gamma_1$ and $\Gamma_2$ are matrices of sizes $(n+2)\times(n+1)$ and $(n+2)\times n$, respectively. 
It includes, in particular, the following theorem, need in the next section. 

\begin{thm} \label{thm:min+1}
Let $n \in \NN$. A near minimal cubature rule of degree $2n-1$ whose number of nodes is equal to
\begin{equation}\label{lwbd+1}
  N = \dim \Pi_{n-1}^2 + \left \lfloor \frac{n}{2} \right \rfloor + 1 = N_{\min}(n) +1
\end{equation}
exists if its nodes are common zeros of $\lfloor \frac{n+1}2 \rfloor$ orthogonal polynomials of 
degree $n$ and $\QQ_n$  in \eqref{quasiOP} for some $\Gamma_1$ and $\Gamma_2$.
\end{thm}

The number of nodes of the cubature rule in \eqref{lwbd+1} is one more than the lower bound given in 
\eqref{lwbd}. Since $x_i P$, $P \in \CV_n(W)$ and $i=1$ or $2$, is $(2n-1)$-orthogonal by the three-term 
relation of orthogonal polynomials in two variables \cite{DX} and, furthermore, it is of degree $n+1$ and of the 
form $c^T \QQ_n$ for some nonzero vector $c \in \RR^{n+2}$, we see that $x_i$ multiples of those 
$\lfloor \frac{n+1}2 \rfloor$ orthogonal polynomials of degree $n$ can be elements of $\QQ_n$. In fact, a quick 
count shows that,  if $n$ is odd, there is only one polynomial in $\QQ_n$ that does not arise this way. It is 
worth to point out that if a cubature rule attains the lower bound \eqref{lwbd}, then its nodes are common zeros
of $\lfloor \frac{n+1}2 \rfloor+1$ orthogonal polynomials of degree $n$, whose $x_i$ multiples lead to a 
$\QQ_n$, as in \eqref{quasiOP}, that vanishes on all nodes. 

The cubature rule with the number of nodes \eqref{lwbd+1} can be obtained by integrating the Lagrange 
interpolation polynomial based on its nodes. The polynomial that interpolates at the nodes of the cubature rule 
arises from the subspace of polynomials $\Pi_n^2 \setminus \CI_n$, where $\CI_n: = \{P_{k,n}: k =1,2,\ldots, 
\lfloor \frac{n+1}{2} \rfloor\}$ and $P_{k,n}$ are the orthogonal polynomials of degree $n$ that vanish on the 
nodes of the cubature rule. Let $K_n(W;\cdot,\cdot)$ be the reproducing kernel of the space
$\CV_n(W)$ and let $\{Q_{k,n}: 0 \le k \le \lfloor \frac{n}{2} \rfloor\}$ be an orthonormal basis of the 
orthogonal complement of $\CI_n$ in $\CV_n(W)$. Then $Q_{k,n} \in \CV_n(W)$ and none of $Q_{k,n}$ vanishes 
on all nodes of the cubature rule. The following theorem is proved in \cite{X94}: 

\begin{thm} \label{near-m-cuba+interp}
Let $W$ be a central symmetric weight function. Let $\{(x_k, y_k): 1 \le k \le N\}$ and $\l_k$ be the nodes and 
weights of the cubature rule of degree $2n-1$ in Theorem \ref{thm:min+1}, respectively. Then
\begin{enumerate}[\rm (a)]
\item There exists a sequence of positive numbers $\{b_{k,n}: 0 \le k \le \left \lfloor \frac{n}{2} \right \rfloor\}$, 
uniquely determined, such that the kernel $K_n^*(\cdot, \cdot)$, defined by 
\begin{equation} \label{Kn*}
 K_n^*(W; x,y) := K_{n-1}(W; x,y) + \sum_{k=0}^{\lfloor \frac{n}{2} \rfloor} b_{k,n} Q_{k,n} (x) Q_{k,n}(y)
\end{equation}
for $x = (x_1,x_2)$ and $y = (y_1,y_2)$, satisfies 
\begin{equation}\label{mcfWeight}
   \l_{k,n} =  \left[ K_n^* (W; (x_k,y_k), (x_k,y_k)) \right ]^{-1}, \quad 1 \le k \le N.
\end{equation}
\item The polynomial $\CL_n f$, defined by
$$
   \CL_n f(x,y) := \sum_{k=1}^N f(x_k,y_k) \ell_k(x,y), \quad \ell_k(x,y):=\lambda_{k,n} K_n^*(W; (x,y), (x_k,y_k)),
$$
satisfies $\CL_n f (x_{k},y_{k}) =  f(x_{k},y_{k})$ for $1 \le k \le N$ and it is the unique interpolation polynomial 
in the space $\Pi_n^*: = \Pi_{n-1}^2 \cup\{Q_{k,n}: 0 \le k \le \lfloor \frac{n}{2} \rfloor\}$.
\end{enumerate}
\end{thm}

Similar statement holds for minimal cubature rules that attain the lower bound \eqref{lwbd}, which is used in 
\cite{X12}. 

\subsection{Orthogonal polynomials with respect to $\CW_{\a,\b,\pm \f12}$}
The orthogonal polynomials with respect to $\CW_{\a,\b,\pm \f12}$ are related to orthogonal polynomials
with respect to the weight function 
\begin{equation} \label{Wabc}
W_{\a,\b,\pm \f12}(u,v) : = b_{\a,\b,\pm \f12}(1-u+v)^\a (1+u+v)^\b (u^2- 4 v)^{\pm \f12},
\end{equation}
where $\a, \b > -1$ and $b_{\a,\b,\pm \f12}$ is the normalization constant so that the integral of $W_{\a,\b,\pm \f12}$
over the domain 
\begin{equation} \label{Omega}
       \Omega: = \{(u,v): 1+u+v > 0, 1-u+v > 0, u^2 > 4v\}, 
\end{equation}
bounded by a parabola and two lines, is 1. Let $p_n^{(\a,\b)}$ be the orthonormal Jacobi polynomials with 
respect to the normalized Jacobi weight function. Then 
an orthonormal basis for $\CV_n(W_{\a,\b,-\frac12})$ is given by 
\begin{equation} \label{OP-1/2}
    P_{k,n}^{\a,\b,-\frac12} (u,v) = \begin{cases} 
              p_n^{(\a,\b)}(x) p_k^{(\a,\b)}(y) + p_n^{(\a,\b)}(y) p_k^{(\a,\b)}(x), & 0 \le k < n, \\ 
               \sqrt{2}  p_n^{(\a,\b)}(x)  p_n^{(\a,\b)}(y), & k =n,\end{cases} 
\end{equation}
and an orthonormal basis for $\CV_n(W_{\a,\b,\frac12})$ is given by 
\begin{equation} \label{OP+1/2}
    P_{k,n}^{\a,\b,\frac12} (u,v) = \frac{ p_{n+1}^{(\a,\b)} (x)  p_k^{(\a,\b)}(y) - p_{n+1}^{(\a,\b)} (y)  p_k^{(\a,\b)}(x)}{x-y}, 
       \quad 0 \le k \le n, 
\end{equation}
where $(u,v)$ and $(x,y)$ are related by $u: = x+y$, $v:= xy$ (\cite{K74}). 
 
The weight function $\CW_{\a,\b,\pm \f12}$ is related to $W_{\a,\b,\pm \f12}$, which we redefine as
\begin{align} \label{CWgamma}
        \CW_{\a,\b,\pm \f12} (x,y) := &\, W_{\a,\b,\pm\f12} (2 x y, x^2+y^2 -1)|x^2-y^2|\\
            = &\, b_{\a,\b,\pm\f12} 2^{\pm 1} |x-y|^{2\a +1}  |x + y|^{2\b +1}   (1-x^2)^{\pm \f12} (1-y^2)^{\pm \f12} \notag
\end{align}
for $(x,y) \in [-1,1]^2$, where we have included the normalization constant, which can be verified by the integral 
relation
\begin{align} \label{Int-P-Q}
 \int_{\Omega} f(u,v) W_{\a,\b,\pm\f12} (u,v) du dv  = \int_{[-1,1]^2} f(2xy, x^2+y^2 -1) \CW_{\a,\b,\pm\f12} (x,y) dx dy.
\end{align}
The orthogonal polynomials with respect to $\CW_{\a,\b,\pm \f12}$ can be expressed in terms of orthogonal polynomials 
with respect to $W_{\a,\b,\pm \f12}$ (\cite{X12a}).  

\begin{prop} \label{prop:OP}
An orthonormal basis for $\CV_{2n}(W_{\a,\b,\pm \f12})$ is given by 
\begin{align*} %\label{Qeven}
   {}_1Q_{k,2n}^{\a,\b,\pm \f12}(x,y):= & P_{k,n}^{\a,\b,\pm \f12}(2xy, x^2+y^2 -1), \quad 0 \le k \le n, \\
   {}_2Q_{k,2n}^{\a,\b,\pm \f12}(x,y) := & a_{\a,\b,\pm \f12}^{(1,1)}(x^2-y^2)  P_{k,n-1}^{\a+1,\b+1,\pm \f12}(2xy, x^2+y^2 -1),  
   \quad 0 \le k \le n-1, 
\end{align*}
and an orthonormal basis for for $\CV_{2n+1}(W_{\a,\b,\pm \f12})$ is given by 
\begin{align*} 
   {}_1Q_{k,2n+1}^{\a,\b,\pm \f12}(x,y):= & a_{\a,\b,\pm \f12}^{(0,1)}(x+y)P_{k,n}^{\a,\b+1,\pm \f12}(2xy, x^2+y^2 -1), 
     \quad 0 \le k \le n, \\
   {}_2Q_{k,2n+1}^{\a,\b,\pm \f12}(x,y) := & a_{\a,\b,\pm \f12}^{(1,0)}(x-y)P_{k,n-1}^{\a+1,\b,\pm \f12}(2xy, x^2+y^2 -1),  
   \quad 0 \le k \le n, 
\end{align*}
where $a_{\a,\b,\pm \f12}^{(i,j)}= \sqrt{b_{\a+i,\b+j,\pm \f12}/b_{\a,\b,\pm \f12}}$. 
\end{prop}

In particular, using the explicit expression of $b_{\a,\b,-\f12}$ in \cite[(2.8) and (2.13)]{X12}, we have 
\begin{align} \label{eq:constant}
\begin{split}
  a_{\a,\b,-\f12}^{(0,1)} & = \frac{\a+\b+2}{2(\b+1)}, \qquad  a_{\a,\b,-\f12}^{(1,0)} = \frac{\a+\b+2}{2(\a+1)}, \\
  a_{\a,\b,-\f12}^{(1,1)} &= \frac{(\a+\b+2)(\a+\b+3)}{4(\a+1)(\b+1)}. 
\end{split}
\end{align}

Using \eqref{OP-1/2} and \eqref{OP+1/2}, it is easy to see that, if $x = \cos \t$, $y = \cos \phi$, then
\begin{align*} %\label{P-1/2-angle}
 & P_{k,n}^{\a,\b, -\frac12} (2 x y, x^2+y^2 -1)  \\ 
 &  \, =   p_n^{(\a,\b)}(\cos (\t - \phi)) p_k^{(\a,\b)}(\cos (\t+\phi)) 
           +  p_k^{(\a,\b)} (\cos (\t - \phi)) p_n^{(\a,\b)}(\cos (\t+\phi)), 
\end{align*}     
where the right hand needs to be divided by $\sqrt{2}$ if $k=n$, and 
\begin{align*} 
 & P_{k,n}^{\a,\b, \frac12} (2 x y, x^2+y^2 -1)  \\ 
 &  \, =  \g_n \frac{p_{n+1}^{(\a,\b)}(\cos (\t - \phi)) p_k^{(\a,\b)}(\cos (\t+\phi)) 
           -  p_k^{(\a,\b)} (\cos (\t - \phi)) p_{n+1}^{(\a,\b)}(\cos (\t+\phi))}{2 \sin \t \sin \phi},
\end{align*}
where $\g_n^2 = b_{\a,\b, - \f12}/(2 b_{\a,\b,\f12})$. These relations will be useed extensively below. 

\section{Minimal and near minimal cubature rules}
\setcounter{equation}{0}

It is convenient to adopt the language of idea and variety in the discussion below. Let $X$ be a set of points. Let
$\CI$ be the polynomial ideal defined by $\CI = \CI(X)= \{p \in \RR[x]: p(x) = 0, \forall x \in X\}$ that has $X$ as its variety.  
A cubature rule of degree $2n-1$ exists if $\CI(X)$ is generated by $(2n-1)$-orthogonal polynomials. Furthermore, 
the existence of a cubature rule can be characterized by the existence of a polynomial idea, generated by 
$(2n-1)$-orthogonal polynomials, such that its variety $X$ consists of finite many real points and 
$|X| = \mathrm{codim}\, \CI $. 

For the weight function $\CW_{\a,\b,\pm \f12}$ a cubature rule that attains the lower bound \eqref{lwbd} is necessarily
a minimal cubature rule and the set of its nodes is the variety of the ideal generated by $\lfloor \f{n+1}2 \rfloor+1$
orthogonal polynomials of degree $n$. For $n= 2m$, it is shown in \cite{X12} that the minimal cubature rule of degree 
$2n-1=4m-1$ exists and its set of nodes is the variety of the ideal $\la {}_1Q_{k,2m}^{\a,\b,\pm \f12}: 0 \le k \le m\ra$. 
For $n=2m+1$, the variety of the idea $\la {}_1Q_{k,2m+1}^{\a,\b,\pm \f12}: 0 \le k \le m\ra$ is infinite, since
it contains the set $\{(x,y): x=y\}$, so is the idea $\la {}_2Q_{k,2m+1}^{\a,\b,\pm \f12}: 0 \le k \le m\ra$. In this case,
more work is needed. 

\begin{thm} \label{thm:minCuba}
For $n = 2m+1$, there exist cubature rules of degree $2n-1 = 4m+1$ for $\CW_{\a,\b, \pm \f12}$ whose number
of nodes is equal to the lower bound $N_{\min}(n)$.
\end{thm}

\begin{proof}
We first consider the case $\CW_{\a,\b,-\f12}$. By Theorem \ref{thm:min}, we need to find $m+2$ orthogonal polynomials
of degree $n$ so that they have $N_{\min}(n)$ many real common zeros. Let $\cos \t_{k,m}$, $1 \le k \le m$, be the zeros 
of the Jacobi polynomial $P_m^{(\a+1,\b)}$. Then $\t_{k,m}^{(\a+1,\b)}:=\t_{k,m} \in (0, \pi)$. Define 
$s_{j,k}^{(\a+1,\b)}:=  s_{j,k}$ and $t_{j,k}^{(\a+1,\b)}:=  t_{j,k}$ by
$$
  s_{j,k} = \cos \tfrac{\t_{j,m} - \t_{k,m}}2, \quad  t_{j,k} = \cos \tfrac{\t_{j,m} + \t_{k,m}}2, \quad 1 \le j \le k \le m.
$$
It is easy to see that the variety of the ideal $\la {}_2Q_{k,2m+1}^{\a,\b,-\f12}: 0 \le k \le m\ra$ is equal to 
$X_m  \cup \{(x,y): x=y\}$, where $X_m: = X_m^{(\a+1,\b)}$ is defined by 
$$
X_m^{(\a+1,\b)}: = \{(s_{j,k}, t_{j,k}), (t_{j,k}, s_{j,k}), (-s_{j,k}, -t_{j,k}), (-t_{j,k}, -s_{j,k}), \, 1 \le j \le k \le m\},
$$
and the cardinality of $X_m$ is $|X_m| = 4 m(m+1)/2 = 2 m(m+1)$. and
$X_m \cap \{(x,y):x=y\} = \emptyset$. We now construct one more orthogonal
polynomial of degree $n$ from $ {}_1Q_{k,2m+1}^{\a,\b, -\f12}$. By the orthogonality of the Jacobi polynomials, 
$$
p_n^{(\a+1,\b)}(t) = \sum_{k=0}^n d_k p_k^{(\a,\b+1)}(t), \quad d_k =
   \int_{-1}^1 p_n^{(\a+1,\b)}(t)p_k^{(\a,\b+1)}(t)w_{\a,\b+1}^*(t)dt,
$$
where $w_{\a,\b}(t): = (1-t)^\a(1+t)^\b$ is the Jacobi weight function and $w_{\a,\b}^*(t) = c_{\a,\b} w_{\a,\b}(t)$ 
is normalized so that that $\int_{-1}^1 w_{\a,\b}^*(t) dt =1$. Hence, taking the linear combination of 
${}_1Q_{k,2m+1}^{\a,\b, -\f12}$, we see that the polynomial
\begin{align*}
  q_m (x,y) = (x+y) & \left[p_m^{(\a,\b+1)}(\cos (\t-\phi))p_m^{(\a+1,\b)}(\cos (\t+\phi)) \right. \\
          & \left. + p_m^{(\a,\b+1)}(\cos (\t+\phi))  p_m^{(\a+1,\b)}(\cos (\t-\phi))\right]
\end{align*}
is an element of $\CV_{2m+1}(\CW_{\a,\b,-\f12})$. By its construction, $q_m$ vanishes on $X_m^{\a,\b+1} 
\cup X_m^{\b+1, \a}$. Furthermore, as we shall see in Lemma \ref{lem:Jacobi} below, the polynomial  
\begin{align*}
 q_m(x,x) = &  2 x \left[p_m^{(\a,\b+1)}(1)p_m^{(\a+1,\b)}(\cos 2\t) + p_m^{(\a,\b+1)}(\cos 2\t)  p_m^{(\a+1,\b)}(1)\right]\\
 = &  2 x \left[p_m^{(\a,\b+1)}(1)p_m^{(\a+1,\b)}(2x^2-1) + p_m^{(\a,\b+1)}(2 x^2-1) p_m^{(\a+1,\b)}(1)\right]
\end{align*}
has $2m+1$ distinct zeros, so that $q_m(x,y)$ has $2m+1$ zeros of the form $(\xi_{j,m}, \xi_{j,m})$, of which 
one is $(0,0)$. Since ${}_2Q_{k,2m+1}^{\a,\b,-\f12}$ contains a factor $x-y$, we see it also has 
$\{(\xi_{j,m},\xi_{j,m}): 0 \le j \le m\}$ as zeros. Consequently, the ideal 
$
 \la q_m^{(\a,\b)},  {}_2Q_{k,2m+1}^{\a,\b,-\f12}: 0 \le k \le m\ra
$
has a variety of zero dimension that is of of a cardinality 
$$
2m(m+1) + 2m = 2 (m+1)^2 -1 = N_{\min}(2m+1).
$$
This completes the proof for $\CW_{\a,\b,-\f12}$. 

The proof for the weight function $\CW_{\a,\b,\f12}$ is similar.
Indeed, it is easy to see that the ideal $\la {}_2Q_{k,2m+1}^{\a,\b, \f12}: 0 \le k \le m\ra$ is equal to 
$X_m  \cup \{(x,y): x= - y\}$, where $X_m$ is defined as above, but with $1 \le j < k \le m+1$ and $s_{j,k}, 
t_{j,k}$ defined in terms of $\t_{k,m+1}^{(\a+1,\b)}$ for $1 \le k \le m+1$. The additional orthogonal polynomial
in the polynomial idea is defined by the $(x-y)$ multiple of 
\begin{align*}
\frac{p_{m+1}^{(\a,\b+1)}(\cos (\t-\phi))p_{m+1}^{(\a+1,\b)}(\cos (\t+\phi))
   - p_{m+1}^{(\a,\b+1)}(\cos (\t+\phi))  p_{m+1}^{(\a+1,\b)}(\cos (\t-\phi))}{\sin \t\sin \phi},
\end{align*}
where we have used the fact that $P_{k,m}^{\a,\b,-1/2}(s,t) =0$ if $k = m+1$. 
\end{proof}

We still need to state and prove the following lemma: 

\begin{lem}\label{lem:Jacobi}
For $\a,\b > -1$ and $m =1,2,\ldots$, 
\begin{enumerate}[ \quad (1)]
 \item The Jacobi polynomials $P_m^{(\a,\b+1)}$ and $P_m^{(\a+1,\b)}$ have no common zeros.
 \item The polynomial $P_m^{(\a,\b+1)}(1)P_m^{(\a+1,\b)}(2t^2-1) + P_m^{(\a+1,\b)}(1) P_m^{(\a,\b+1)}(2t^2-1)$ has
 $2m$ simple zeros inside $(-1,1)$.  
\end{enumerate}
\end{lem}

\begin{proof}
Assuming $x_0$ is a common zero of $P_m^{(\a+1,\b)}$ and $P_m^{(\a,\b+1)}$. Then, by the identity
\cite[(22.7.17)]{AS}
\begin{align} \label{eq:Jacobi1}
  (1-t) P_m^{(\a+1,\b)}(t)+(1+t) P_m^{(\a,\b+1)}(t)= 2 P_{m}^{(\a,\b)}(t),
\end{align}
we see that $x_0$ is a zero of $P_{m}^{(\a,\b)}$ and, by the identity \cite[(22.7.20)]{AS}
\begin{align} \label{eq:Jacobi3}
   P_m^{(\a+1,\b)}(t) - P_m^{(\a,\b+1)}(t) = P_{m-1}^{(\a+1,\b+1)}(t),
\end{align}
$x_0$ is also a zero of $P_{m-1}^{(\a+1,\b+1)}$ or a zero of the derivative of $P_m^{(\a,\b)}$, which contradicts 
the fact that the zeros of the Jacobi polynomials are simple. This proves (1). 

For (2), let $q_m(t): = P_m^{(\a,\b+1)}(1)P_m^{(\a+1,\b)}(t) + P_m^{(\a+1,\b)}(1) P_m^{(\a,\b+1)}(t)$. We need two more identities 
of the Jacobi polynomials \cite[(22.7.18) and (22.7.19)]{AS} 
\begin{align}\label{eq:Jacobi2}
\begin{split}
  P_{m}^{(\a+1,\b+1)}(t) + \frac{m+\a+1}{m+\g} P_{m-1}^{(\a+1,\b+1)}(t) & =
     \frac{2m+\g}{m+\g} P_m^{(\a+1,\b)}(t),\\
  P_{m}^{(\a+1,\b+1)}(t) - \frac{m+\b+1}{m+\g} P_{m-1}^{(\a+1,\b+1)}(t) & =  
    \frac{2m+\g}{m+\g} P_m^{(\a,\b+1)}(t),
\end{split}
\end{align}
where $\g = \a+\b+2$. Together with $P_m^{(\a+1,\b)}(1) = \binom{m+\a+1}{m} = \frac{m+\a+1}{\a+1} P_m^{(\a,\b+1)}(1)$, 
we can then rewrite $q_m(t)$ as 
\begin{align*}
  q_m(t)  = & \frac{(m+\g)(m+2\a+2)}{(2m+\g)\a+1} \binom{m+\a}{m} \\ 
     & \times  \left[P_m^{(\a+1,\b+1)}(t) - 
      \frac{(m+\a+1)(m+\b-\a)}{(m+\g)(m+2\a+2)} P_{m-1}^{(\a+1,\b+1)}(t) \right]. 
\end{align*}
By its definition, $q_m(1) > 0$ and $(-1)^m q_m(-1) > 0$, which are sufficient to imply that the quasi-orthogonal 
polynomial in the bracket has $n$ simple zeros in $(-1,1)$ (see, e.g., \cite[Theorem 5.3]{X94b}). 
\end{proof}

It should be mentioned that the existence of these cubature rules are known only for $\CW_{-\f12, -\f12, \f12}(x,y) =
\sqrt{1-x^2}\sqrt{1-y^2}$, the product Chebyshev weight of the second kind (\cite{MP}, it is new even for 
$\CW_{-\f12, -\f12, \f12}(x,y) = 1/(\sqrt{1-x^2}\sqrt{1-y^2})$, the product Chebyshev weight of the first kind.

\begin{rem}
The minimal cubature rules are not unique. Indeed, exchanging the role of $\a$ and $\b$, starting with
${}_1Q_{k,2m+1}^{\a,\b,-\f12}$ instead of ${}_2Q_{k,2m+1}^{\a,\b,-\f12}$, we can obtain another cubature 
rule of degree $2n-1$ with $N_{\min}(n)$ nodes for $n=2m+1$, which has $2m$ nodes on the diagonal 
$x+y=0$ of the square $[-1,1]^2$ rather than on the diagonal $x-y=0$. 
\end{rem}

We depict the nodes of minimal cubature rules of degree $4m+1$ in Figure 1 for the weigh function
$$
  \frac{1}{\sqrt{1-x^2}\sqrt{1-y^2}} \quad \hbox{and} \quad \frac{|x-y|^2 |x+y|^2}{\sqrt{1-x^2}\sqrt{1-y^2}}.
$$ 
\begin{figure} [ht] 
\includegraphics[scale=0.45]{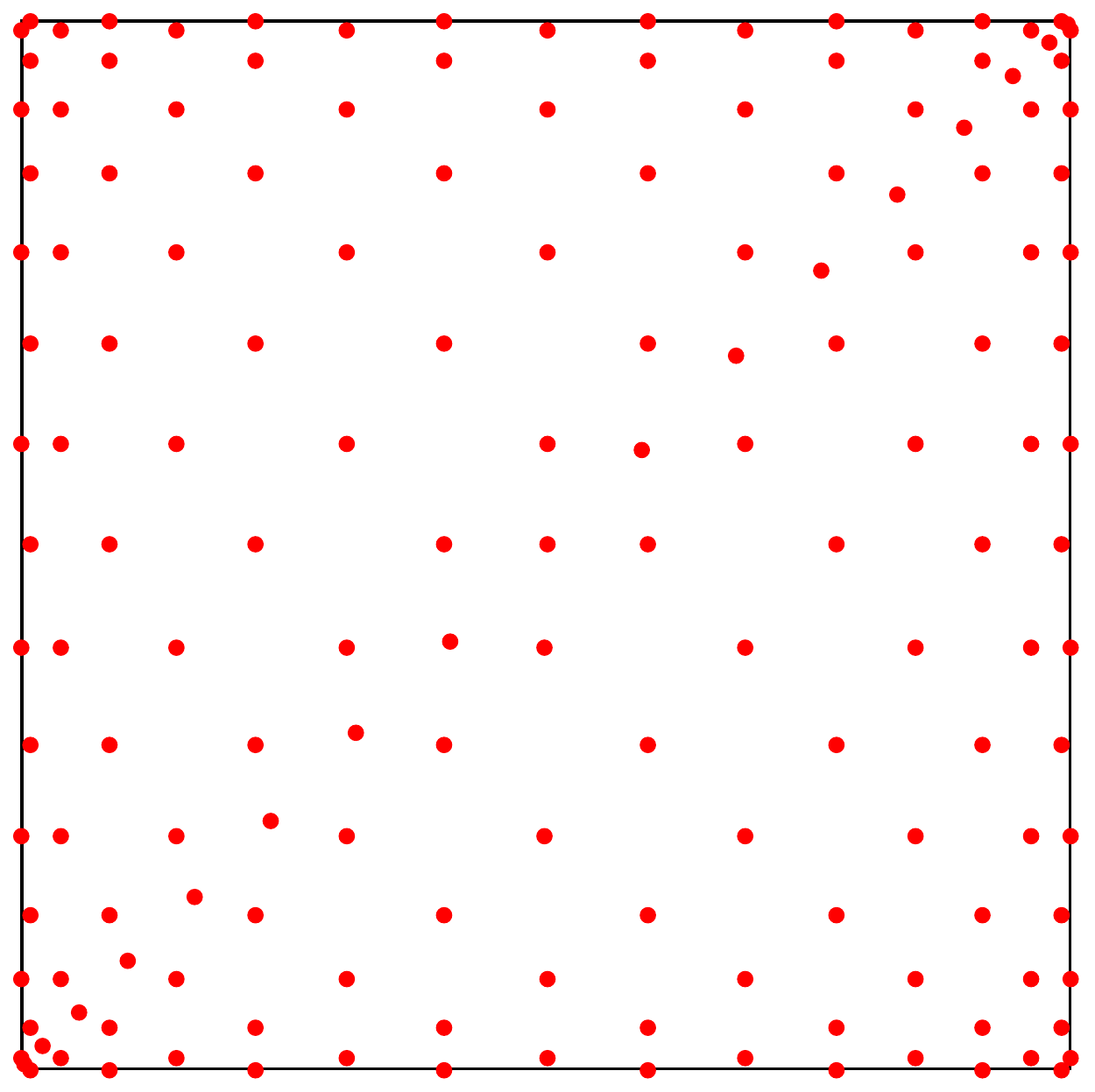} \qquad \includegraphics[scale=0.45]{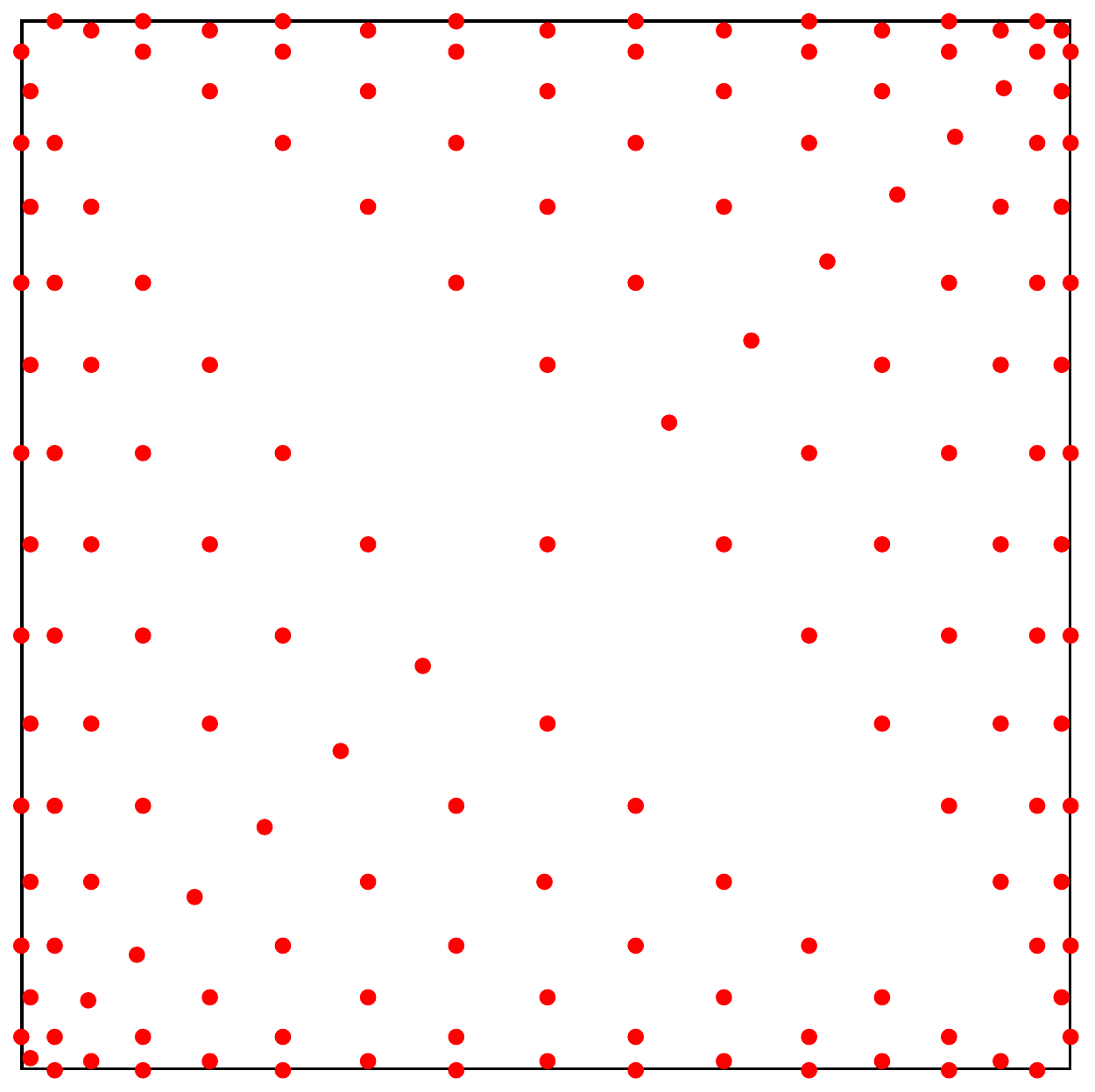}
\caption{Nodes of minimal cubature rules of degree 33 for $\CW_{-\frac12, -\frac12, - \frac12}$ and
   $\CW_{\f12,\f12,- \frac12}$ }.
\end{figure} 

Since the nodes of the minimal cubature rule are explicitly given in the proof of Theorem \ref{thm:minCuba}, 
the coefficients $\l_{k}$ of the cubature rule \eqref{eq:CF} can be found by solving a linear system of moment 
equations. However, it does not look to be easy to give these coefficients explicitly. In the following, we construct 
explicitly another set of cubature rules of degree $4m+1$ for $\CW_{\a,\b, \pm \f12}$, which are near minimal 
with $N_{\min}(2m+1)+1$ 
number of nodes, just one more than that of the minimal cubature rules. These cubature rules are derived from 
the Gauss-Radau quadrature rule for the Jacobi weight and, in this regard, they are similar to the construction 
for the case of $n=2m$ in \cite{X12}, which is based on the Gauss quadrature rules for the Jacobi weight. 
 
For $\a,\b > -1$, recall that  $w_{\a,\b}^*(t) = c_{\a,\b} w_{\a,\b}(t)$ is the normalized Jacobi weight function
and that the Gaussian quadrature rule of degree $2n-1$ for $w_{\a,\b}$ takes the form
$$
   \int_{-1}^1 f(t) w_{\a,\b}^*(t) dt = \sum_{k=1}^n \l_{k,n}^{(\a,\b)} f(x_{k,n}^{(\a,\b)}),
$$ 
where $x_{k,n}^{(\a,\b)} = \cos \t_{k,n}^{(\a,\b)}$, $1 \le k \le n$, are the zeros of the Jacobi polynomial $P_n^{(\a,\b)}$
and $\l_{k,n}^{(\a,\b)}$ are evaluations of the Christoffel function for $w_{\a,\b}$ at $x_{k,n}^{(\a,\b)}$. The 
Gauss-Radau quadrature rules for the weight $w_{\a,\b}$ are given in the following proposition.  

\begin{prop}
For every polynomial $f$ of degree at most $2n$, 
\begin{equation} \label{Gauss-Radau}
  \int_{-1}^1 f(t) w_{\a,\b}^*(t) dt = \mu_{0,n}^{(\a,\b)} f(1)+ \sum_{k=1}^n \mu_{k,n}^{(\a,\b)} f(x_{k,n}^{(\a+1,\b)}), 
\end{equation}
where 
$$
\mu_{0,n}^{(\a,\b)} = \frac{(\b+1)_n}{\binom{n+\a+1}{n}(\a+\b+1)_n}, \quad 
\mu_{k,n}^{(\a,\b)} =  \frac{2(\a+1)}{\a+\b+2} \cdot \frac{\l_{k,n}^{(\a+1,\b)}}{1-t^{(\a+1,\b)}}, \quad 1 \le k \le n.
$$
%and $\mu_{0,n}^{(\a,\b)} = 1 -  \frac{2(\a+1)}{\a+\b+2} \sum_{k=1}^n \mu_{k,n}^{(\a,\b)} > 0$.
\end{prop} 

The existence of this quadrature rule is well-known. Explicit formulas for $\mu_{k,n}^{(\a,\b)}$ can be found 
in \cite[(3.10)]{G}, where $\a$ and $\b$ need to be exchanged and also multiplying by $c_{\a,\b}$. The 
formula of $\mu_{k,n}^{(\a,\b)}$ stated in the proposition, which will be needed later, can be derived as follows: 
writing $f(x) = (1-x) g(x) +f(1)$, where $g$ is a polynomial of degree at most $2n-1$, and applying the 
Gaussian quadrature rule with respect to $w_{\a+1,\b}$, we see that 
\begin{align*}
 \int_{-1}^1 f(x) w_{\a,\b}^*(x) dx   = & c_{\a,\b} \int_{-1}^1 g(x) w_{\a+1,\b}(x) dx + f(1) \\
  = & \frac{c_{\a,\b}}{c_{\a+1,\b}} \sum_{k=1}^n \l_{k,n}^{(\a+1,\b)} g(x_{k,n}^{(\a+1,\b)}) + f(1) \\
  = &\frac{c_{\a,\b}}{c_{\a+1,\b}}  \sum_{k=1}^n  \frac{\l_{k,n}^{(\a+1,\b)}}{1-t^{(\a+1,\b)}}
     \left(f(x_{k,n}^{(\a+1,\b)}) - f(1) \right) + f(1)
\end{align*}
for all $f$ of degree at most $2n$. The constant $c_{\a,\b}$ can be easily computed, so is $c_{\a,\b}/c_{\a+1,\b}$. 
This establishes the formula for $\mu_{k,n}^{(\a,\b)}$.

Let $\CW_{\a,\b,\pm \f12}^*: = b_{\a,\b,\pm \f12}\CW_{\a,\b,\pm \f12}$ with $b_{\a,\b,\pm \f12}$ being the 
normalization constant so that $\int_{[-1,1]^2} \CW^*_{\a,\b,\pm \f12}(x,y) dxdy =1$. For $m =0,1,\dots$, let 
$\t_{k,m} = \t_{k,m}^{(\a+1,\b)}$, $1 \le k \le m$, as before and we define $\t_{0,m} : =0$. We choose the order
of these number so that $0 = \t_{0,m} < \t_{1,m} < \ldots < \t_{m,m} < \pi$. Let 
\begin{align} \label{stjk}
  s_{j,k}: =  \cos \tfrac{\t_{j,m}-\t_{k,m}}{2} \quad\hbox{and}\quad
                t_{j,k} := \cos \tfrac{\t_{j,m} + \t_{k,m}}{2}, \quad 0 \le j \le k \le m,
\end{align}
where $ s_{j,k}= s_{j,k}^{(\a+1,\b)}$ and $t_{j,k}= t_{j,k}^{(\a+1,\b)}$. These are the same as we defined in
Theorem \ref{thm:minCuba}, but here we allow $j=0$ with $\t_{0,m} =0$. 

\begin{thm} \label{thm:cubaCW}
For $\CW_{\a,\b, -\frac12}$ on $[-1,1]^2$, we have the near minimal cubature rule of degree $4m+1$ 
with $N_{\min}(2m+1)+1$ nodes, 
\begin{align} \label{NearMin-}
    \int_{[-1,1]^2} &  f(x, y) \CW_{\a,\b, -\frac12}^* (x,y) dx dy  =
     \frac12 \sum_{k=0}^m \mathop{ {\sum}' }_{j=0}^k  \mu_{k,m}^{(\a,\b)} \mu_{j,m}^{(\a,\b)} \\ 
   & \times \left[ f( s_{j,k}, t_{j,k})+  f( t_{j,k}, s_{j,k}) +   f( - s_{j,k}, - t_{j,k})+  f(- t_{j,k},- s_{j,k})  \right],  \notag 
\end{align}
where $ {\sum}' $ means that the terms for $j =k$ are divided by 2. 
For $\CW_{\a,\b,\frac12}$ on $[-1,1]^2$, we have the minimal cubature rule of degree $4m-3$ with 
$N_{\min}(2m-1)+1$ nodes, 
\begin{align} \label{NearMin+}
    \int_{[-1,1]^2} &  f(x, y) \CW_{\a,\b, \frac12}^* (x,y) dx dy  =
     \frac12 \sum_{k=0}^m \mathop{ {\sum}' }_{j=0}^k  \mu_{k,m}^{(\a,\b)} \mu_{j,m}^{(\a,\b)} (\cos \t_{j,n} - \cos \t_{k,n})^2
      \notag \\ 
   & \times \left[ f( s_{j,k}, t_{j,k})+  f( t_{j,k}, s_{j,k}) +   f( - s_{j,k}, - t_{j,k})+  f(- t_{j,k},- s_{j,k})  \right]. 
\end{align}
\end{thm}

\begin{proof}
We consider $\CW_{\a,\b, -\frac12}$ first. Let $x_{k,m} = \cos \t_{j,m}^{(\a+1,\b)}$. Applying the product  
Gaussian-Radau cubature rule
$$
 \int_{[-1,1]^2}f(x,y)w_{\a,\b}^*(x)w_{\a,\b}^*(y) dxdy 
   = \sum_{j=0}^m \sum_{k=0}^m \mu_{j,m}^{(\a,\b)} \mu_{k,m}^{(\a,\b)} f(x_{j,m},x_{k,m}),
$$
which holds for $f \in \Pi^2_{2m} \times \Pi_{2m}^2$, on the functions $f(x+y,xy)$, it follows, by symmetry and 
changing variables $u= x+y$ and $v=xy$ on $\{(x,y)\in [-1,1]^2: y > x\}$, that
\begin{align} \label{productCR}
 \int_\Omega f(u,v) W_{\a,\b, -\f12}^*(x,y) dxdy = 2 \sum_{k=0}^m \mathop{ {\sum}' }_{j=0}^k \mu_{j,m}^{(\a,\b)} 
  \mu_{k,m}^{(\a,\b)} f(u_{j,k},v_{j,k}),
\end{align}
where $u_{j,k} = x_{j,m}+x_{k,m}$ and $v_{j,k} = x_{j,m} x_{k,m}$, and this cubature rule is known to hold for
all polynomials of degree up to $2m$ \cite{SX}. Consequently, by \eqref{Int-P-Q}, we see that 
$$
 \int_{[-1,1]^2} f(2xy,x^2+y^2-1) \CW_{\a,\b,-\f12}^*(x,y) dxdy = 2
    \sum_{k=0}^m \mathop{ {\sum}' }_{j=0}^k  \mu_{j,n}^{(\a,\b)} \mu_{k,n}^{(\a,\b)} f(u_{j,k},v_{j,k})
$$
for all $f \in \Pi_{2m}^2$. It is easy to verify that $u_{j,k} = \cos\t_{j,m}+\cos\t_{k,m} = 2 s_{j,k} t_{j,k}$ and
$v_{j,k} = \cos\t_{j,m}\cos\t_{k,m} = s_{j,k}^2+ t_{j,k}^2-1$. Hence, by the Sobolev theorem on invariant cubature 
rules \cite{Sobolev}, which states that a cubature rule invariant under a finite group $G$ is exact for all elements 
of a subspace $\CP$ of polynomials if, and only if, it is exact for all polynomials in $\CP$ that are invariant under 
$G$, the last cubature rule implies \eqref{NearMin-} for $f\in \Pi_{4m}^2$. Furthermore, since 
the weight function $\CW_{\a,\b,-\f12}$ is invariant under $(x,y) \mapsto (-x,-y)$ and so is the right hand side of
\eqref{NearMin-}, it follows that the cubature \eqref{NearMin-} holds for $f \in \Pi_{4m+1}$.  When $j = 0$, 
$s_{0,k} = t_{0,k}$, so that $(t_{0,k}, s_{0,k}) = (s_{0,k},t_{0,k})$ for $0 \le k \le m$. Hence, the number of nodes
of the cubature rule in \eqref{NearMin-} is 
$$
2 m(m+1) + 2(m+1) = 2(m+1)^2 = N_{\min}(2m+1)+1.
$$
This completes the proof in the case of $\CW_{\a,\b,-\f12}$. The proof for \eqref{NearMin+} can be carried out
similarly, starting from applying the product cubature rule \eqref{productCR} on $(x-y) f(x+y, xy)$, as in \cite{X12}. 
We skip the details.
\end{proof}

In the case of the product Chebyshev weight function $\CW_{-\frac12,-\frac12,-\frac12}$, the cubature rule
\eqref{NearMin-} is established in \cite{X94} and later in \cite{BP, LSX} using two other methods. In view of 
Theorem \ref{thm:min+1}, one natural question is what are the polynomials that vanish on the nodes of such 
a cubature rule; in other words, what is the polynomial idea whose variety is the set of nodes of such a 
cubature rule. The idea is identified in the proof of the following theorem. 

\begin{thm}
For $n = 2m+1$, the set of nodes of the cubature rules of degree $2n-1=4m+1$ in \eqref{NearMin-} is the variety of
the polynomial idea generated by ${}_2Q_{k,2m+1}^{\a,\b,-\f12}$, $0 \le k \le m$, and $(2n-1)$-orthogonal
polynomials of the form $\QQ_{2m+2} + \Gamma \QQ_{2m}$, where 
$\QQ_{2m} = \{{}_1Q_{k,2m}^{\a,\b,-\f12}, 0 \le k \le m\}\cup \{{}_2Q_{k,2m}^{\a,\b,-\f12}, 0 \le k \le m-1\}$.
\end{thm}

\begin{proof}
Let us denote by $X_m$ the set of nodes of the cubature rule in  \eqref{NearMin-}, 
$$
X_m: = \{(s_{j,k}, t_{j,k}), (t_{j,k}, s_{j,k}), (-s_{j,k}, -t_{j,k}), (-t_{j,k}, -s_{j,k}), 0 \le j \le k \le m\}.
$$
By the definition of $s_{j,k} = s_{j,k}^{(\a+1,\b)}$ and $t_{j,k}=t_{j,k}^{(\a+1,\b)}$ in \eqref{stjk}, it follows immediately
from the explicit formulas of orthogonal polynomials in Proposition \ref{prop:OP} that ${}_2Q_{k,2m+1}^{\a,\b,-\f12}$, 
$0 \le k \le m$, vanish on $X_m$, where we need the factor $x-y$ in the polynomials when $j=0$ since 
$s_{0,k} = t_{0,k}$. These account for $\lfloor \f{n+1}2 \rfloor = m+1$ orthogonal polynomials of degree $n$ specified
in Theorem \ref{thm:min+1}. Now, from the first of the following two identities \cite[(22.7.15) and (22.7.16)]{AS} 
\begin{align*} %\label{eq:Jacobi1}
\begin{split}
(1+t) P_m^{(\a,\b+1)}(t)&\ = \frac{2(m+1)}{2m+\a+\b+2} P_{m+1}^{(\a,\b)}(t) + \frac{2(m+\b+1)}{2m+\a+\b+2} P_m^{(\a,\b)}(t),\\
(1-t) P_m^{(\a+1,\b)}(t)&\ = \frac{-2(m+1)}{2m+\a+\b+2} P_{m+1}^{(\a,\b)}(t) + \frac{2(m+\a+1)}{2m+\a+\b+2} P_m^{(\a,\b)}(t),
\end{split}
\end{align*}
it follow readily that, for $0 \le k \le m$,
\begin{align*}
& {}_1Q_{k,2m+2}^{\a,\b,-\f12}(x,y) - a_{k,m} \f{m+\a+1}{m+1}  {}_1Q_{k,2m}^{\a,\b,-\f12}(x,y) \\ 
&\quad  \doteq  (1-\cos (\t - \phi)) P_{m}^{(\a+1,\b)}(\cos (\t-\phi)) P_k^{(\a,\b)}(\cos (\t+\phi))  \\
   & \qquad + (1-\cos (\t +\phi)) P_m^{(\a+1,\b)}(\cos (\t+\phi)) P_k^{(\a,\b)}(\cos (\t-\phi)),  
\end{align*}
where $a_{k,m}$ is a constant coming from the normalization constants for the Jacobi polynomials and 
$\doteq$ means that the identity holds under a constant multiple, whereas for $k = m+1$, we have
\begin{align*}
& {}_1Q_{m+1,2m+2}^{\a,\b,-\f12}(x,y) - a_{m+1,m} \f{m+\a+1}{2(m+1)} {}_1Q_{m,2m+2}^{\a,\b,-\f12}(x,y) \\ 
&\quad  \doteq  P_{m+1}^{(\a,\b)}(\cos (\t-\phi))  (1-\cos (\t - \phi))P_m^{(\a+1,\b)}(\cos (\t+\phi))  \\
   & \qquad + (1-\cos (\t - \phi)) P_{m+1}^{(\a,\b)}(\cos (\t+\phi)) P_m^{(\a+1,\b)}(\cos (\t-\phi)), 
\end{align*}
from which it is easy to see that these polynomials vanish on $X_m$. Furthermore, using the first identity of
\eqref{eq:Jacobi2}, it is easy to see that, for $0 \le k \le m$, 
\begin{align*}
& {}_2Q_{k,2m+2}^{\a,\b,-\f12}(x,y) + a_{k,m} \f{m+\a+1}{m+1}  {}_1Q_{k,2m}^{\a,\b,-\f12}(x,y) \\ 
&\quad  \doteq (x^2-y^2)  \left[ P_{m}^{(\a+1,\b)}(\cos (\t-\phi)) P_{m+1}^{(\a,\b)}(\cos (\t+\phi)) \right. \\
   & \qquad \left. + P_{m-1}^{(\a+1,\b)}(\cos (\t+\phi)) P_{m+1}^{(\a,\b)}(\cos (\t-\phi))\right],  
\end{align*}
where $a_{k,m}$ is another constant coming from normalization of the Jacobi polynomials, whereas for $k = m+1$,
we have
\begin{align*}
& {}_2Q_{m+1,2m+2}^{\a,\b,-\f12}(x,y) - a_m \f{m+\a+1}{2(m+\a+\b+2)} {}_1Q_{m,2m+2}^{\a,\b,-\f12}(x,y) \\ 
&\quad  \doteq  (x^2-y^2) \left[ P_{m}^{(\a+1,\b+1)}(\cos (\t-\phi)) P_m^{(\a+1,\b)}(\cos (\t+\phi)) \right. \\
   & \qquad \quad\qquad\qquad \left. + P_{m}^{(\a+1,\b+1)}(\cos (\t+\phi)) P_m^{(\a+1,\b)}(\cos (\t-\phi)) \right], 
\end{align*}
from which it follows that these polynomials vanish on $X_m$ as well. Together, this shows that, with $n = 2m+1$, 
there are $n+2$ many $(2n-1)$-orthogonal polynomials of the form $P_{n+1}+ a P_{n-1}$, where $P_k\in 
\CV_k(\CW_{\a,\b})$. Together, they can be written in the form $\QQ_{n+1} + \Gamma \QQ_{n-1}$. This completes 
the proof. 
\end{proof}

\begin{rem}
By symmetry, we also obtain another cubature rule of degree $4m+1$ by reversing the role of $\a$ and $\b$. More 
precisely, the nodes of the corresponding cubature rule is $X_m$ with $s_{j,k}$ and $t_{j,k}$ given in terms of the 
zeros of the Jacobi polynomials of $P_n^{(\a,\b+1)}$. 
\end{rem}

We depict the nodes of cubature rules \eqref{NearMin-} of degree $4m+1$ in Figure 2. Apart from the nodes on the 
diagonal $x =y$, the distribution of these nodes are similar to those for the minimal cubature rules of degree $4m-1$ 
(see Figure 1 in \cite{X12}). 

\begin{figure} [ht] 
\includegraphics[scale=0.45]{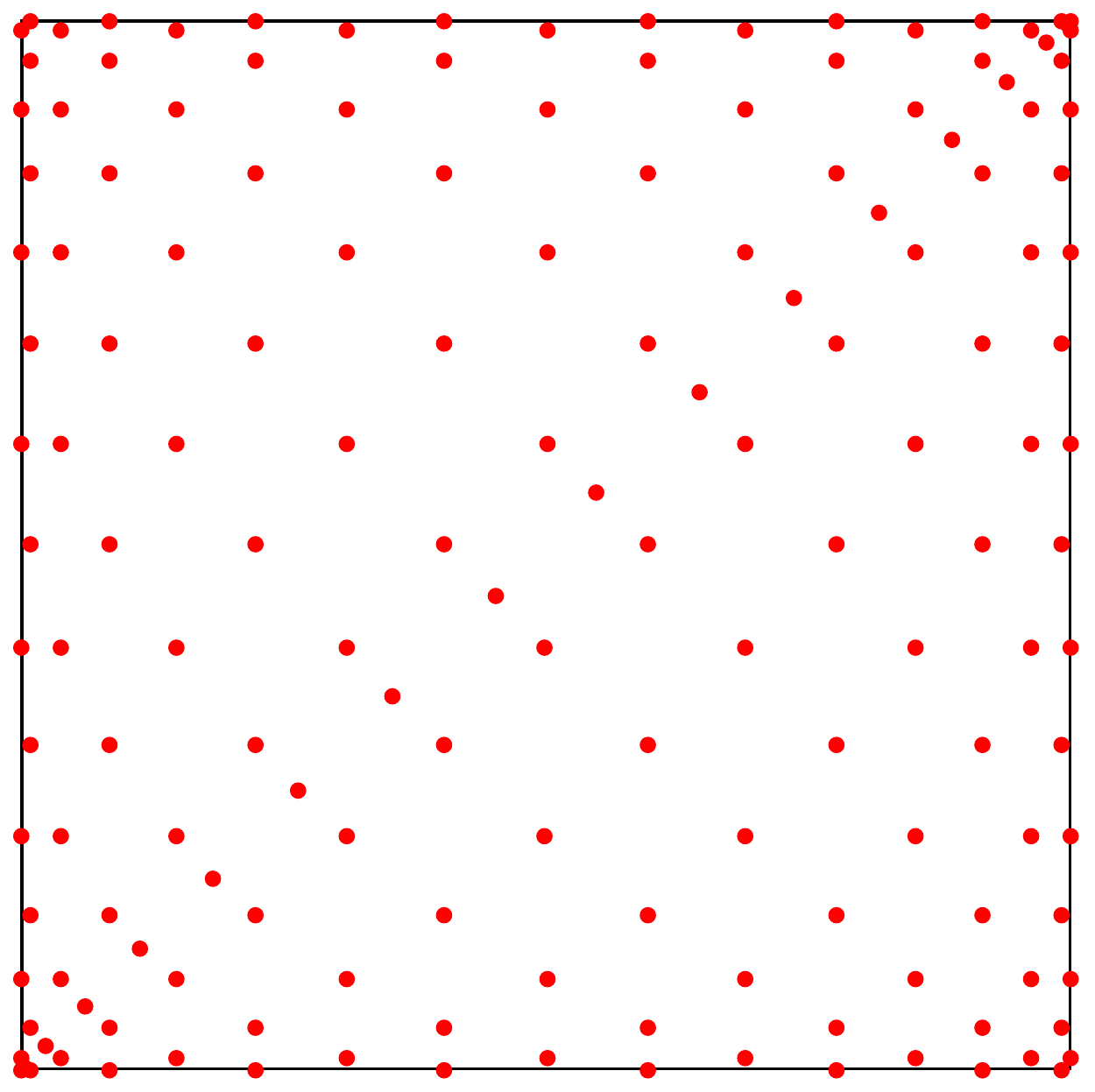} \qquad \includegraphics[scale=0.45]{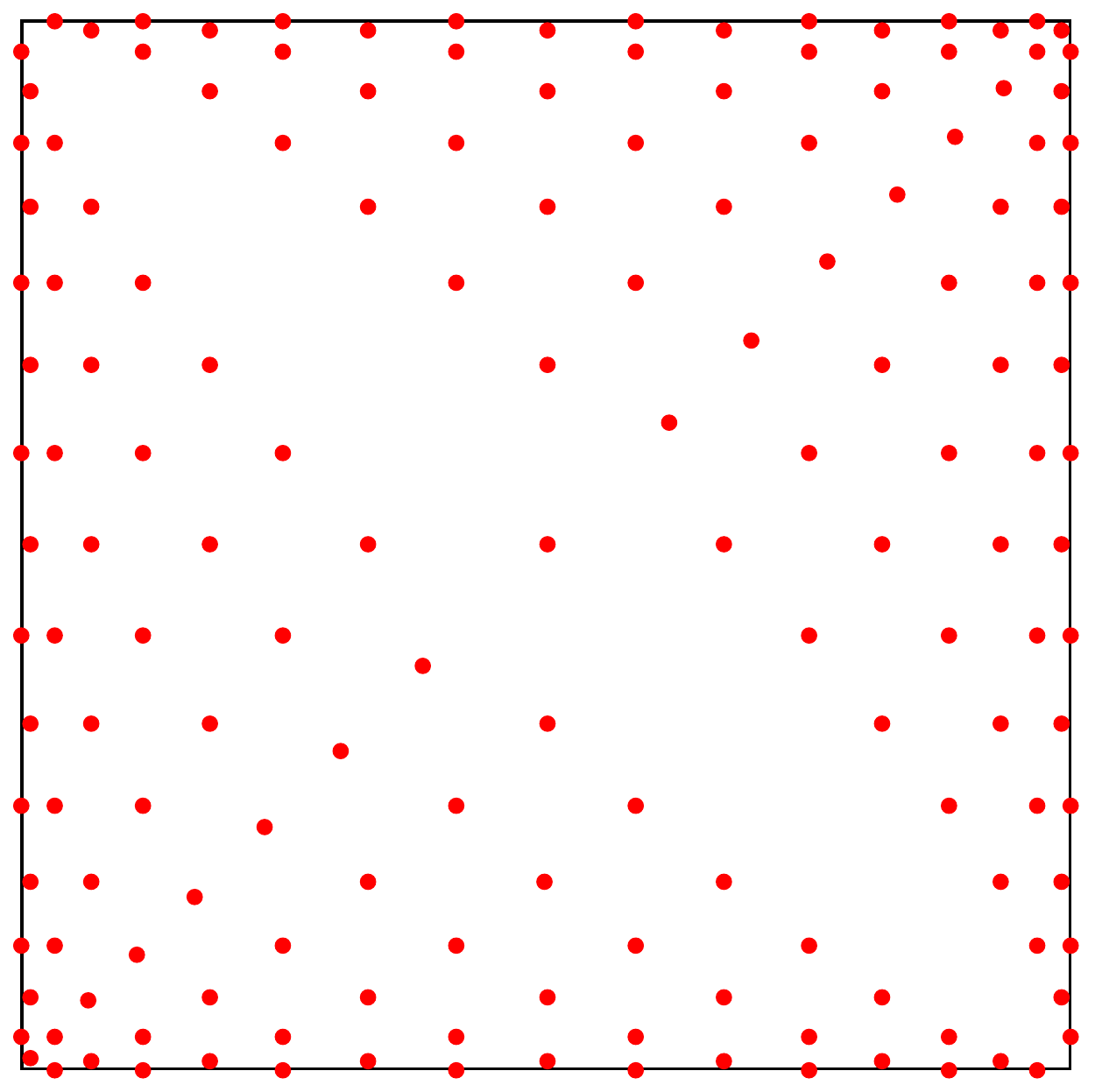}
\caption{Nodes of minimal cubature rules of degree 33 for $\CW_{-\frac12, -\frac12, - \frac12}$ and
   $\CW_{\f12,\f12,- \frac12}$ }
\end{figure} 
 
Just as in the case of $n =2m$ in \cite{X12}, the nodes appear to be propelled away from the diagonals of the
square, which is where the singularity appears in $|x-y|^{2\a+1}|x+y|^{2\b+1}$ of the weight function $\CW_{\a,\b,-\f12}$. 
We can in fact determine this region precisely. Let $\cos \t_{k,m}^{(\a,\b)}$ be the zeros of the Jacobi polynomial 
$P_m^{(\a,\b)}$ as before with the order
$
   0 < \t_{1,m}^{(\a,\b)} < \t_{2,m}^{(\a,\b)} < \cdots < \t_{m,m}^{(\a,\b)} < \pi.
$

\begin{prop}
Let $\a, \b > -1$. For $0< \t_{1,m}< \t_{2,m} < \ldots< \t_{m,m} < \pi$, define parametric curves 
\begin{align*}
    S_1^{\pm}:=& \{ \pm (\cos \t, \cos (\t-\t_{1,m})): \t_{1,m} < \t < \pi\}, \\
    S_2^{\pm} :=& \{\pm (\cos \t, \cos (\t-\t_{m,m})): 0 < \t < \t_{m,m}\}.
\end{align*}
\begin{enumerate}
\item For the cubature rule of degree $4m+1$ in \eqref{NearMin-}, let $\t_{j,m} = \t_{j,m}^{(\a+1,\b)}$; then
all nodes are in the region bounded by $S_1^{\pm}$ and $S_2^{\pm}$ and the boundary of $[-1,1]^2$ 
except those on the diagonal $x= y$. 
\item For the minimal cubature rule of degree $4m$ in \cite[(4.2)]{X12}, let $\t_{j,m} = \t_{j,m}^{(\a,\b)}$; then
all nodes are in the region bounded by $S_1^{\pm}$ and $S_2^{\pm}$  and the boundary of $[-1,1]^2$.
\end{enumerate}
\end{prop}

\begin{proof}
Let $\t_{k,m} = \t_{k,m}^{(\a,\b)}$ and define $\xi_{j,k} = (\t_{j,m} + \t_{k,m})/2$, $\eta_{j,k} = (\t_{j,m} - \t_{k,m})/2$. 
Then $t_{j,k} = \cos \xi_{j,k}$ and $s_{j,k} = \cos \eta_{j,k}$. The arguments of the nodes in $X_m^{(1)}: = 
\{(t_{j,k},s_{j,k}), 1 \le j \le k \le m\}$ consist of the set $\{(\xi_{j,k},\eta_{j,k}): 1 \le j \le k \le m\}$, which is roughly
a triangle with $(\xi_{m,m}, \eta_{m,m}), \ldots, (\xi_{2,2}, \eta_{2,2}), (\xi_{1,1}, \eta_{1,1})$ as one side
and $(\xi_{1,m}, \eta_{1,m})$ as the vertex opposite to that side. The two other sides of the triangle are  
$(\xi_{j,m},\eta_{j,m})$, $1 \le j \le m$, and $(\xi_{1,k},\eta_{1,k})$, $1 \le k \le m$, respectively. Since 
\begin{align*}
  \xi_{j,m} & = (\t_{j,m}+\t_{m,m})/2 = (\t_{j,m} - \t_{m,m})/2+ \t_{m,m} = \eta_{j,m} + \t_{m,m},\\
  \xi_{1,k} & = (\t_{1,m}+\t_{k,m})/2 = - (\t_{1,m} - \t_{k,m})/2+ \t_{1,m} = -\eta_{1,k} + \t_{1,m},
\end{align*}
we see that $X_m^{(1)}$ is bounded by the parametric curve $(\cos \t, \cos (\t - \t_{m,m}))$ and 
$(\cos \t, \cos (\t- \t_{1,m}))$, or the curves $S_1^+$ and $S_2^+$. The other three subsets of the nodes can be 
considered similarly. 
\end{proof}

As an illustration, we depict the curves for the nodes of the two cubature rules in the above proposition in
Figure 3 for the case $\a=\b=1/2$, or the weight function
$$
  W_{\f12,\f12,-\f12}(x,y) = \frac{(x-y)^2(x+y)^2}{\sqrt{1-x^2}\sqrt{1-y^2}},
$$ 
\begin{figure} [ht] 
\includegraphics[scale=0.45]{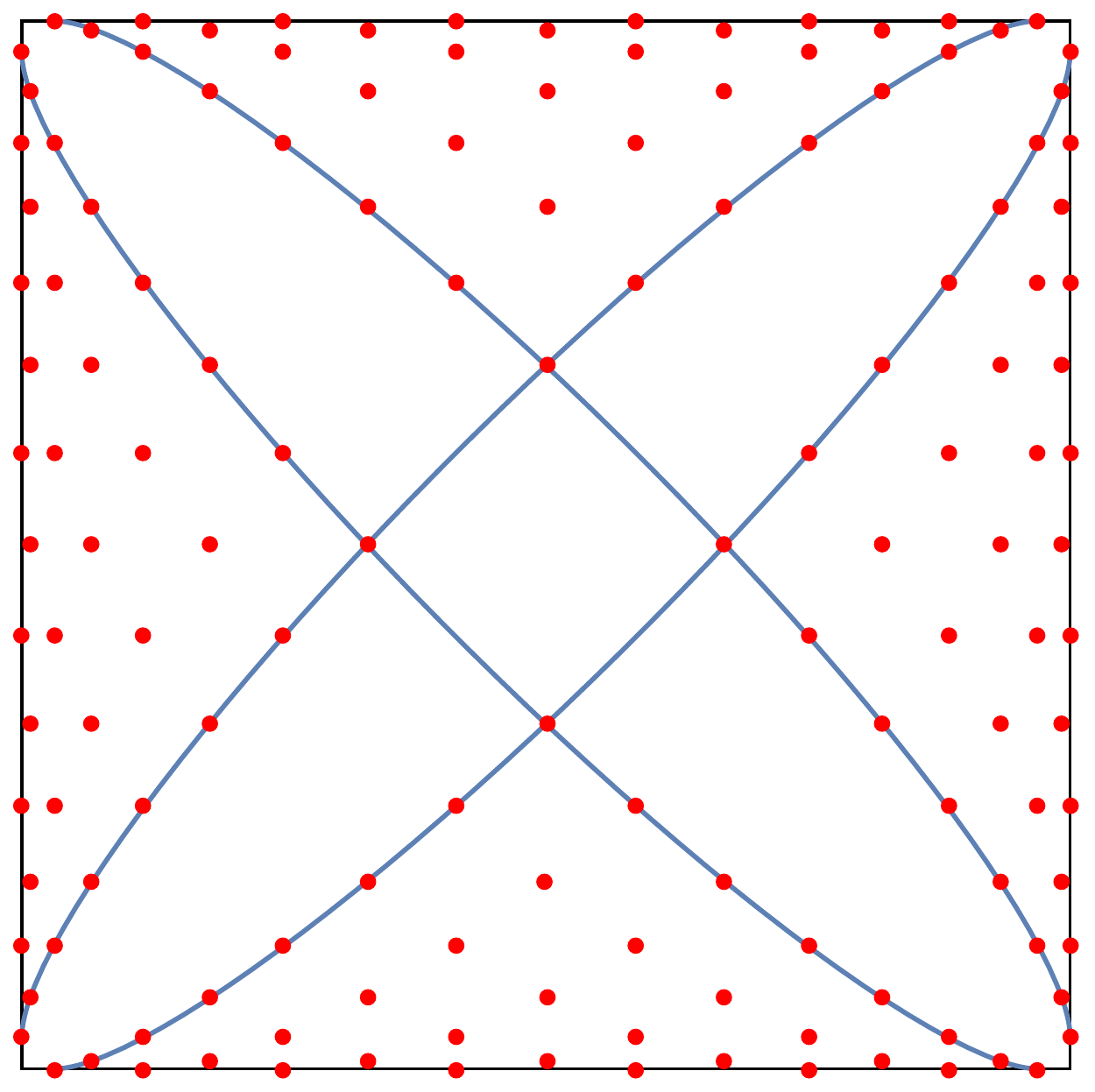} \qquad \includegraphics[scale=0.45]{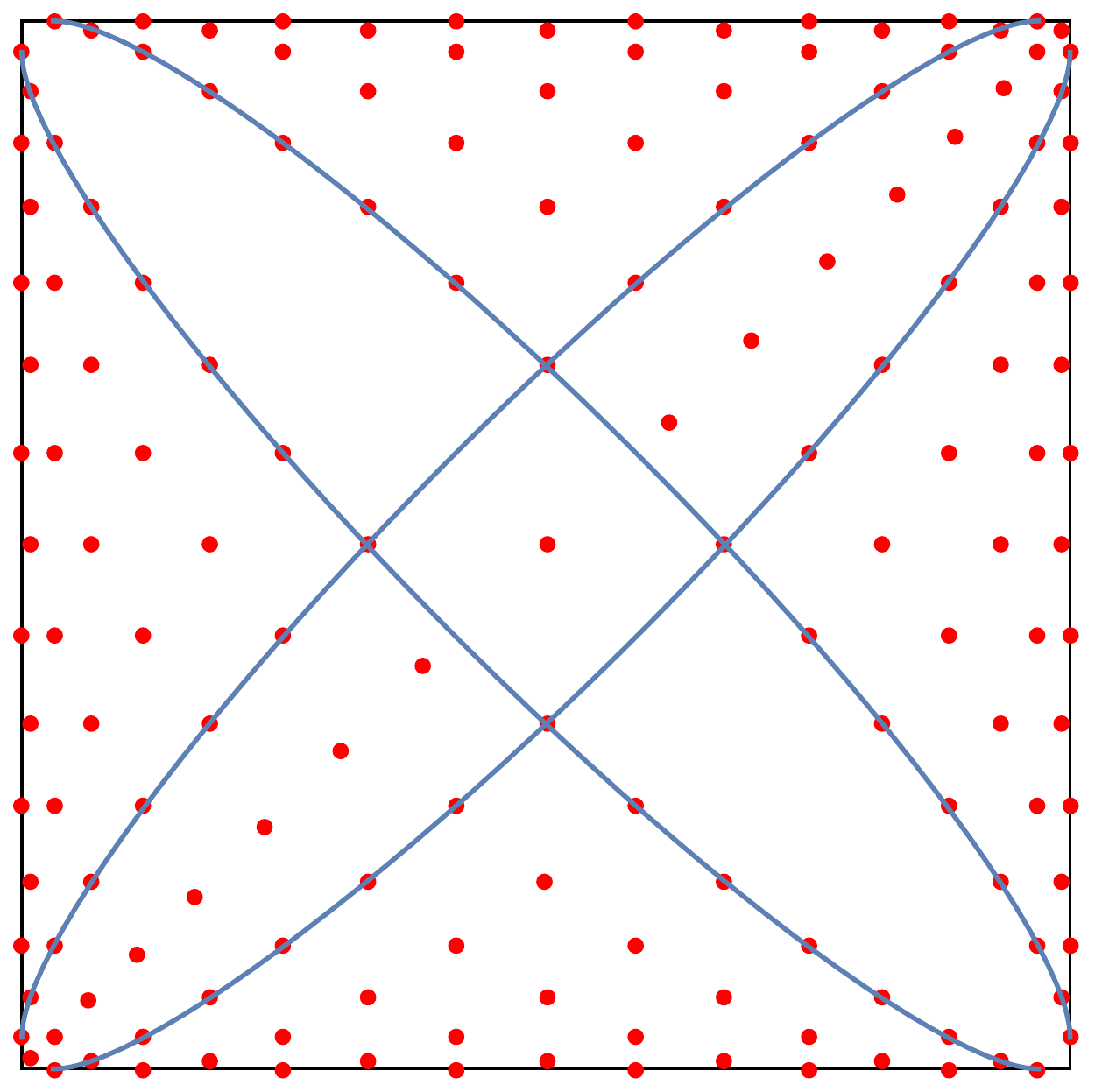}
\caption{Nodes of minimal cubature rule of degree 35 and near minimal rule of degree 33
for $\CW_{\f12,\f12,-\frac12}$} 
\end{figure} 

\noindent
and in Figure 4 for the case $\a = 1/2$ and $\b =-1/2$, or the weight function
$$
  W_{\f12,\f12,-\f12}(x,y) = \frac{(x-y)^2}{\sqrt{1-x^2}\sqrt{1-y^2}}.
$$ 
\begin{figure} [ht] 
\includegraphics[scale=0.45]{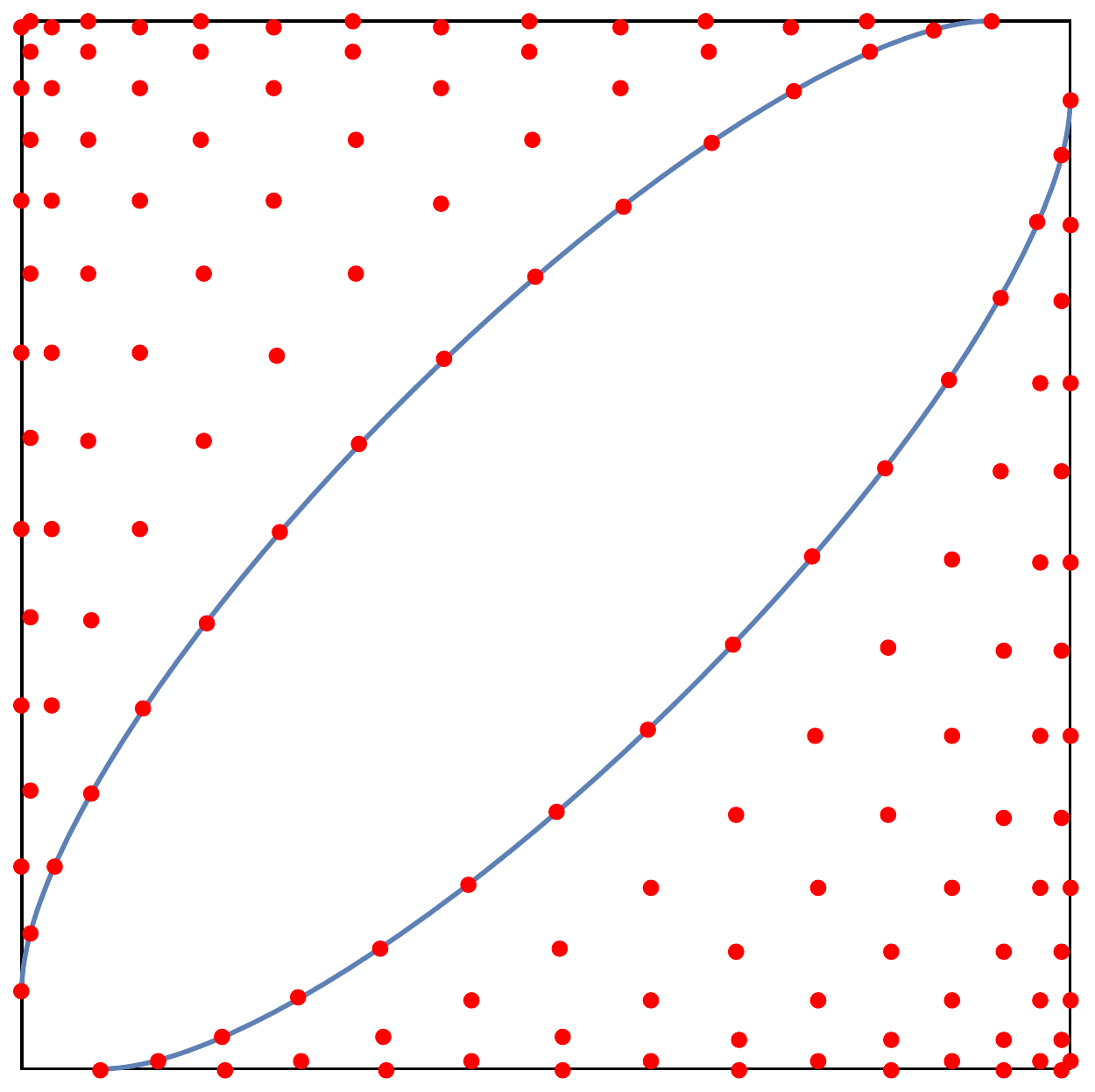} \qquad \includegraphics[scale=0.45]{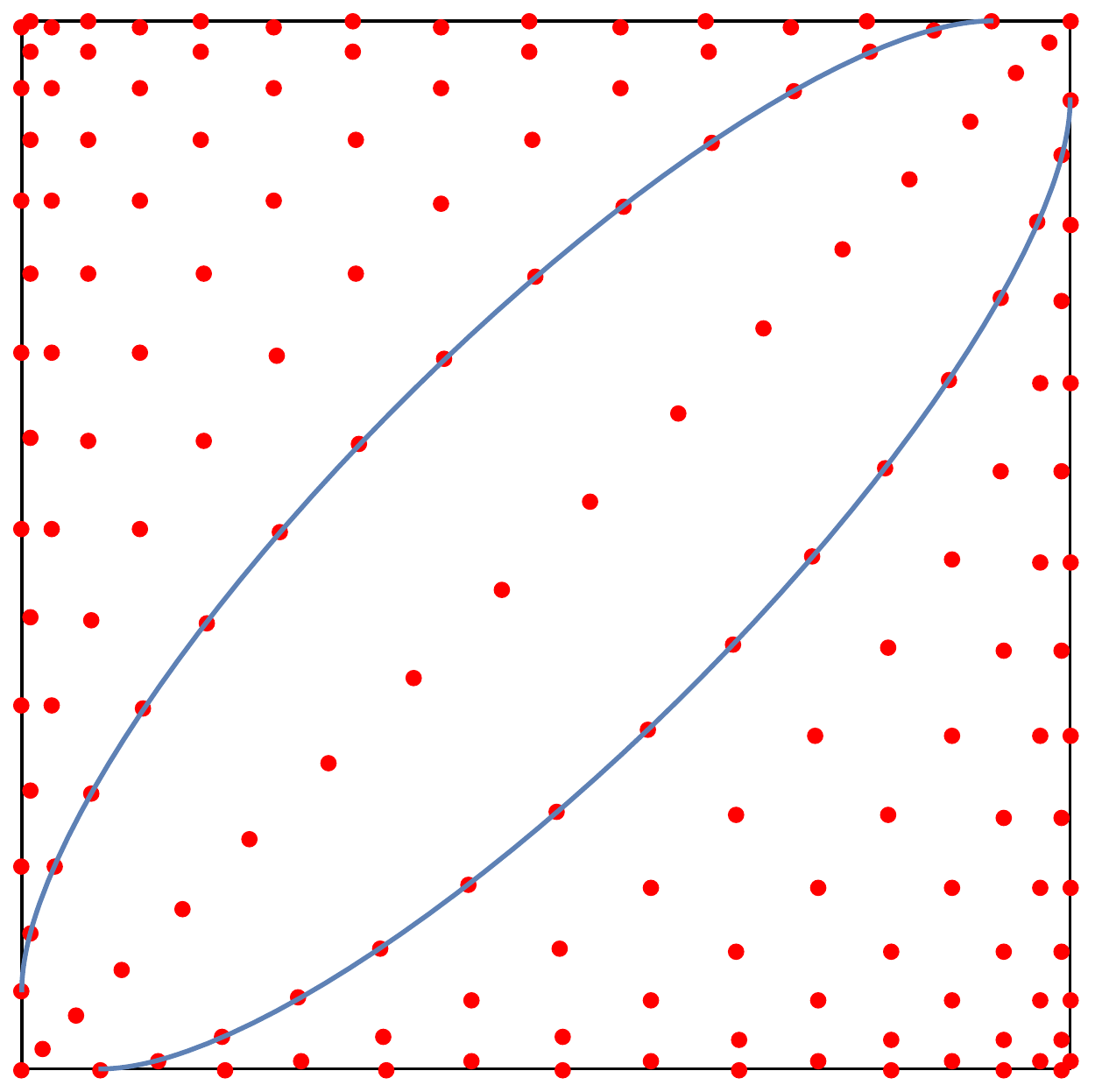}
\caption{Nodes of minimal cubature rule of degree 35 and near minimal rule of degree 33
for $\CW_{\f12,-\f12,-\frac12}$} 
\end{figure} 

If we increase the values of $\a,\b$, then the region around the diagonals that does not contain nodes increase in
size. On the other hand, it is known that $\t_{1,m}^{(\a,\b)} \to 0$ and $\t_{m,m}^{(\a,\b)} \to \pi$ when $m \to \infty$, 
so that the four parametric curves tend to the two diagonals $x-y=0$ and $x+y=0$ of $[-1,1]^2$. In other words, 
the holes in the figures will close down as $m$ increases. 
 
\section{Lagrange interpolation and near minimal cubature rules}
\setcounter{equation}{0}

The near minimal cubature rule in Theorem \ref{thm:cubaCW} can also be obtained by integrating the
Lagrange interpolation polynomial based on its nodes, as we stated in Theorem \ref{near-m-cuba+interp}. 
The interpolation polynomial is unique in the space $\Pi_n^*$ and its explicit formula can be obtained if
we can determine $b_{k,n}$ in \eqref{Kn*}. In this section we consider the case of $\CW_{\a,\b, - \frac12}$.
Throughout this section we write $\CW_{\a,\b}: = \CW_{\a,\b, -\frac12}$ and suppress $-\f12$ in the subscript
and superscript whenever we can; for example, we write ${}_iQ_{k,n}^{\a,\b}$ for ${}_iQ_{k,n}^{\a,\b,-\f12}$. 

\subsection{Construction of the interpolation polynomial}

We need a lemma on the Jacobi polynomials.

\begin{lem}\label{hat-hn}
For $\a,\b > -1$ and $m \ge 0$, let $x_{k,m} = x_{k,m}^{(\a+1,\b)}$ and let $\mu_{k,m}^{(\a,\b)}$ be the 
coefficients in the Gauss-Radau quadrature \eqref{Gauss-Radau}. Define
$$
  \wh h_{\ell}^{(\a,\b)}:= \sum_{k=0}^m \mu_k^{(\a,\b)} (1+x_{k,m})^2 \left[p_\ell^{(\a,\b+1)}(x_{k,m})\right]^2.
$$
Then $\wh h_{\ell}^{(\a,\b)} =  2(\b+1)/(\a+\b+2)$ for $0\le \ell \le m-1$ and 
$$
 \wh h_m^{(\a,\b)} = \frac{(\b+1)(\a+\b+m+1)(\a+\b+2m+2)}{(\a+\b+2)(\b+m+1)(\a+\b+2m+1)}.
$$
\end{lem}

\begin{proof}
If $0 \le \ell \le m-1$, then $(1+x)^2 [p_\ell^{(\a,\b+1)}(x)]^2$ is of degree $2\ell + 2 \le 2m$, so that 
$\wh h_\ell$ is equal to the integral of this polynomial with respect to $w_{\a,\b}$ by the Gauss-Radau
quadrature. Consequently, since $(1+x)^2 w_{\a,\b}(x) = w_{\a,\b+1}(x)$, we see that 
$ \wh h_\ell = c_{\a,\b} /c_{\a,\b+1} = 2(\b+1)/(\a+\b+2)$. The same argument does not work when
$\ell =m$, since the polynomial is of degree $2m+2$. However, by the relation \eqref{eq:Jacobi1},
$(1+x_{k,m})P_m^{(\a,\b+1)}(x_{k,m}) = 2 P_m^{(\a,\b)}(x_{k,m})$ so that we can apply the Gauss-Radau 
quadrature on $P_m^{(\a,\b)}$ to compute $\wh h_m$. Using $p_m^{(\a,\b)}(x) = [h_m^{(\a,\b)}]^{-\f12} 
P_m^{(\a,\b)}(x)$, where $h_m^{(\a,\b)}$ denotes the $L^2(w_{\a,\b})$ norm of $P_m^{(\a,\b)}$ normalized 
as $h_m^{(\a,\b)} = \int_{-1}^1 |P_m^{(\a,\b)}(x)|^2 w_{\a,\b}^*(x)dx$, we conclude that 
$$
\wh  h_m^{(\a,\b)} = \frac{h_m^{(\a,\b)}}{h_m^{(\a,\b+1)}}  \int_{-1}^1 \left[p_m^{(\a,\b)}(x)\right]^2 w_{\a,\b}^*(x) dx =  
\frac{h_m^{(\a,\b)}}{h_m^{(\a,\b+1)}},
$$
from which the stated formula for $\wh h_m^{(\a,\b)}$ follows from that of $h_m^{(\a,\b)}$ (cf. \cite[(4.33)]{Sz}).
\end{proof}

We now give the explicit formula of the interpolation polynomial based on the nodes of the near minimal
cubature rule. Let $\t_{k,m} = \t_{k,m}^{(\a+1,\b)}$ and let $s_{j,k}$ and $t_{j,k}$ be defined as in
\eqref{stjk}. Set 
\begin{align*}
    \xb_{j,k}^{(1)}:=(s_{j,k},t_{j,k}), \,\,  \xb_{j,k}^{(2)} :=(t_{j,k},s_{j,k}), \,\, 
      \xb_{j,k}^{(3)}:=(-s_{j,k},- t_{j,k}), \,\, \xb_{j,k}^{(4)}:= (-t_{j,k},-s_{j,k}). 
\end{align*}
Then the nodes of the near cubature formula for $\CW_{\a,\b}$ in \eqref{NearMin-} consist of 
$$
  X_m: = \{\xb_{j,k}^{(i)}: 0\le j \le k \le m, \,\, i=1,2,3,4\}. 
$$ 
The subspace $\Pi_{2m+1}^*$ in which the interpolation polynomial is unique takes the form
$$ 
 \Pi_{2m+1}^* : = \Pi_{2m}^2 \cup \mathrm{span} \{{}_1Q_{k,2m+1}^{\a,\b+1}: 0 \le k \le n-1\}
$$ 
and the kernel $K_n^*$, in \eqref{Kn*}, used to determine the interpolation polynomials is 
\begin{equation} \label{wtCKn}
  \CK_{2m+1}^* (x,y) : = \CK_{2m}^{\a,\b}(x,y) + 
           \sum_{k=0}^{m} b_{k,m} {}_1Q_{k,2m+1}^{\a,\b+1}(x) {}_1Q_{k,2m+1}^{\a,\b+1}(y),
\end{equation}
where $b_{k,m}$ are certain positive numbers,  ${}_1Q_{k,2m+1}^{\a,\b+1}$ are defined in 
Proposition \ref{prop:OP}, and $\CK_{2m}^{\a,\b}$ is the reproducing kernel of $\Pi_{2m}^2$ in 
$L^2(\CW_{\a,\b})$, given explicitly in \cite[(4.23)]{X12a}. The constants $b_{k,m}$ and the interpolation 
polynomial $\CL_n^{\a,\b}$ are explicitly determined in the following theorem. 

\begin{thm} \label{interp-CL*}
For $n =2m+1$, the Lagrange interpolation polynomial, $\CL_n^{\a,\b} f$, in $\Pi_{2m+1}^*$ that interpolates 
$f$ on $X_m$ is given by 
\begin{align} \label{CLnWab}
   \CL_n^{\a,\b} f(x,y) = \sum_{k=0}^m   \sum_{j=0}^k & \left[ f\left(\xb_{j,k}^{(1)}\right) \ell^{(1)}_{j,k} (x,y)
         + f\left(\xb_{j,k}^{(2)}\right) \ell^{(2)}_{j,k} (x,y) \right. \\
         & \left. + f\left(\xb_{j,k}^{(3)}\right) \ell^{(3)}_{j,k} (x,y)+f\left(\xb_{j,k}^{(4)}\right) \ell^{(4)}_{j,k} (x,y) \right], \, \notag
\end{align}
where the fundamental interpolation polynomials $\ell^{(i)}_{j,k}$ are given by
\begin{align} \label{ell-ijk}
  \ell^{(i)}_{j,k} (x,y) = \frac12 \mu_{j,m}^{(\a,\b)}\mu_{k,m}^{(\a,\b)} \CK_{2m+1}^*\left((x,y), \xb_{j,k}^{(i)}\right),
\end{align}
in which $\frac12$ in the right hand side needs to be replaced by $\frac14$ when $j = k$, and  
\begin{align} \label{eq:CK*Jacobi}
    \CK_{2m+1}^* (x,y) & =     \CK_{2m}^{\a,\b}(x,y) \, + \frac{1}{2 \wh h_0^{(\a,\b)} \wh h_m^{(\a,\b)}} (x_1+x_2)(y_1+y_2)   \\
      & \times  \Big[ K_{m}^{\a,\b+1} (X, Y) - K_{m-1}^{\a,\b+1} (X,Y)    
          \notag \\
      & \left. + \left( \frac{\wh h_0^{(\a,\b)}}{\wh h_m^{(\a,\b)}} - 1\right)
        P_{m,m}^{\a,\b+1,-\f12}(X)P_{m,m}^{\a,\b+1,-\f12}(Y)\right], \notag
\end{align}
where $X= (2 x_1 x_2,x_1^2+x_2^2-1)$, $Y= (2 y_1 y_2, y_1^2+y_2^2 -1)$, $\CK_{2m}^{\a,\b}(\cdot,\cdot)$ is
given in \cite[(4.23)]{X12a} with $\g = -1/2$ and $K_m^{\a,\b}(\cdot,\cdot)$ is given in \cite[(4.24)]{X12a}.
\end{thm}

\begin{proof}
That the interpolation polynomial is given by \eqref{CLnWab} and \eqref{ell-ijk} follows from 
Theorem \ref{near-m-cuba+interp}. What remains to be done is to determine the constants $b_{k,m}$ 
in \eqref{wtCKn}, which will verify the validity of \eqref{eq:CK*Jacobi}. 

We denote by $\C_m[f]$ the near minimal cubature rule, that is,  
$$
   \C_m [f]:= \frac12 \sum_{k=0}^m \mathop{ {\sum}' }_{j=0}^k  
     \mu_{k,m}^{(\a,\b)} \mu_{j,m}^{(\a,\b)}  \left[  f\left(\xb_{j,k}^{(1)}\right)+   f\left(\xb_{j,k}^{(2)}\right) 
          +  f\left(\xb_{j,k}^{(3)}\right)+   f\left(\xb_{j,k}^{(4)} \right) \right].
$$           
Throughout this proof, we write $Q_{k,2m+1}(x,y) = {}_1Q_{k,2m+1}^{\a,\b+1}(x,y)$. By the symmetry of
$Q_{k,2m+1}$, seen fom the explicit formula of $Q_{k,2m+1}$ in Proposition \ref{prop:OP}, we have 
\begin{equation} \label{Qxjk2}
    Q_{\ell,2m+1}\left(\xb_{j,k}^{(1)}\right) =  Q_{\ell,2m+1}\left(\xb_{j,k}^{(2)}\right) 
       =  -Q_{\ell,2m+1}\left(\xb_{j,k}^{(3)}\right) = -Q_{\ell,2m+1}\left(\xb_{j,k}^{(4)}\right). 
\end{equation}
Since $\ell_{j,k}^{(i)}$ are the fundamental interpolation polynomials, we obtain, by \eqref{mcfWeight}, 
\begin{equation*}
     \CK_{2m+1}^*\left(\xb_{j,k}^{(1)}, \xb_{j',k'}^{(1)}\right) = 2 \left(\mu_{j,m}^{(\a,\b)}\mu_{k,m}^{(\a,\b)}\right)^{-1} 
            \delta_{j,j'}\delta_{k,k'}. 
\end{equation*}
The above two displayed equations allow us to show that 
\begin{equation} \label{eqn-bmn}
     \C_n \left[\CK_{2m+1}^*\left(\xb_{j,k}^{(1)}, \cdot\right) Q_{\ell, 2m+1}\right] =  Q_{\ell, 2m+1}\left(\xb_{j,k}^{(1)}\right).
\end{equation}

We now compute the same quantity in another way. Using \eqref{wtCKn} and the fact that the cubature rule is 
of degree $4m+1$, it follows by the cubature rule and the orthogonality of $Q_{\ell,2m+1}$ that 
\begin{align*}
  \C_n \left[\CK_{2m+1}^*\left(\xb_{j,k}^{(1)}, \cdot\right) Q_{\ell, 2m+1}\right] 
        =  \sum_{q=0}^{m} b_{q,m}  \C_n \left[ Q_{q,2m+1} Q_{\ell, 2m+1}\right]
\end{align*}
Using the symmetry of $Q_{\ell,2m+1}$, it follows from \eqref{Qxjk2} that
\begin{align*}
   \C_n \left[ Q_{q,2m+1} Q_{\ell, 2m+1}\right]  = \sum_{k=0}^m \sum_{j=0}^m \mu_{k,m}^{(\a,\b)}\mu_{j,m}^{(\a,\b)}
       Q_{q,2m+1}\left(\xb_{j,k}^{(1)}\right) Q_{\ell,2m+1}\left(\xb_{j,k}^{(1)}\right). 
\end{align*} 
Now, by the definition of $s_{j,k}$ and $t_{j,k}$, we obtain
$$
 s_{j,k} + t_{j,k} = 2 \cos \tfrac{\t_{j,m}}{2}\cos \tfrac{\t_{k,m}}{2} = \sqrt{1+x_{j,m}}\sqrt{1+x_{k,m}},
$$
so that, from the explicit formula of $Q_{k,2m+1}$, 
\begin{align*} %\label{Qxjk}
  &  Q_{\ell,2m+1}\left(\xb_{j,k}^{(1)}\right) =  a_{\a,\b, -\f12}^{(0,1)}  \sqrt{1+x_{j,m}}\sqrt{1+x_{k,m}} \\
     & \quad  \times \left[ p_m^{(\a,\b+1)}(x_{j,m}) p_\ell^{(\a,\b+1)}(x_{k,m})+
                  p_{m}^{(\a,\b+1)}(x_{k,m}) p_\ell^{(\a,\b+1)}(x_{j,m})\right]. \notag
\end{align*}
Consequently, by \eqref{Qxjk2} and Lemma \ref{hat-hn}, we deduce from the Gauss-Radau rule \eqref{Gauss-Radau}
that 
\begin{align*}
 \C_n \left[ Q_{q,2m+1} Q_{\ell, 2m+1}\right]  = 2  \left[a_{\a,\b, -\f12}^{(0,1)} \right]^2 \wh h_\ell \wh h_m \delta_{q,\ell}, 
    \quad 0 \le q, \ell \le m-1.
\end{align*} 
Putting these formulas together, we have shown that 
$$
\C_n \left[\CK_{2m+1}^*\left(\xb_{j,k}^{(1)}, \cdot\right) Q_{\ell, 2m+1}\right] 
    = 2 \left[a_{\a,\b, -\f12}^{(0,1)} \right]^2 \wh h_\ell \wh h_m b_{\ell,m} Q_{\ell, 2m+1}\left(\xb_{j,k}^{(1)}\right). 
$$
Comparing with \eqref{eqn-bmn}, it follows readily that 
$
\left[a_{\a,\b, -\f12}^{(0,1)} \right]^2 b_{\ell, m}   = (2   \wh h_\ell \wh h_m)^{-1}. 
$
In particular,  by Lemma \ref{hat-hn}, we see that $b_{0,m} = \ldots = b_{m-1,m}$, so that 
\begin{align*}
  \sum_{k=0}^m b_{k,m} Q_{k, 2 m+1}(x) Q_{k,2m+1}(y) = &\  b_{0,m} \sum_{k=0}^m Q_{k, 2 m+1}(x) Q_{k,2m+1}(y)\\
   &  +( b_{m,m} - b_{0,m})Q_{m, 2 m+1}(x) Q_{m,2m+1}(y).
\end{align*}
Since $K_m^{\a,\b}$ is the reproducing kernel of $\Pi_m^2$ in $L^2(W_{\a,\b,-\f12})$, it follows from the explicit 
formula of $Q_{k,2m+1} = {}_1Q_{k,2m+1}^{\a,\b,-\f12}$ in Proposition \ref{prop:OP} that 
\begin{align*}
    b_{0,m} \sum_{k=0}^{m} & Q_{k, 2 m+1}(x) Q_{k,2m+1}(y)\\
        & = \frac{1}{2 \wh h_0 \wh h_m} (x_1+x_2)(y_1+y_2)
                   \left[ K_{m}^{\a,\b+1} (s,t) - K_{m-1}^{\a,\b+1} (s,t) \right]. 
\end{align*} 
Putting these in \eqref{wtCKn} and using the explicit formula of $Q_{m, 2 m+1}$, we arrive at 
\eqref{eq:CK*Jacobi}. This completes the proof. 
\end{proof}
 
We present the Lagrange interpolation polynomial in the form of \eqref{eq:CK*Jacobi} because the reproducing 
kernels are fundamental tools for further study and they are explicitly known. As one consequence
of the explicit expression in the theorem, we can state a bound for the Lebesgue constant of the interpolation
operator $\CL_n^{(\a,\b)}$, which is defined as its operator norm $\|\CL_n^{(\a,\b)}\|_\infty$. 
  
\begin{thm}
Let $\a, \b \ge -1/2$. The Lebesgue constant of the Lagrange interpolation polynomial 
$\CL_n^{\a,\b}f$ based on the nodes of the near minimal cubature rule of degree $2n-1$, with
$n = 2m+1$, for $\CW_{\a,\b}$ satisfies 
\begin{equation} \label{LebesgueCLn}
  \|\CL_n^{\a,\b}\|_\infty = \CO(1) \begin{cases} n^{2 \max \{\a,\b\} + 1}, & \max\{\a,\b\} > -1/2, \\
                                (\log n)^2, & \max\{\a,\b\} = -1/2.  \end{cases}
\end{equation}
\end{thm}
 
\begin{proof}
The proof is similar to the estimate carried out in \cite{X12} for $n=2m$, except in one estimate. It is
easy to see that 
\begin{equation}\label{eq:Lebesgue}
\|\CL_n^{\a,\b} f\|_\infty \sim \max_{x \in [-1,1]^2} \sum_{k=0}^n {{\sum}'}_{j=0}^k \mu_{j,m}^{(\a,\b)}\mu_{k,m}^{(\a,\b)}
  \left | \CK_{2m+1}^* (x, \xb_{j,k}^{(1)})\right |. 
\end{equation}
From the explicit formula of $\CK_{2m+1}^*$ in \eqref{eq:CK*Jacobi}, we need to estimate the sum
$$
\Lambda_m:=\sum_{k=0}^n \sum_{j=0}^k \mu_{j,m}^{(\a,\b)}\mu_{k,m}^{(\a,\b)} 
      \left| R_m^{\a,\b}(x_1,x_2) R_m^{\a,\b}(s_{j,k}, t_{j,k}) \right|,
$$
where $R_m^{\a,\b} (x_1,x_2) = (x_1+x_2) P_{m,m}^{\a,\b,-\f12}(2x_1x_2,x_1^2+x_2^2-1)$. Using the explicit
formula of $P_{m,m}^{\a,\b,-\f12}$ in \eqref{OP-1/2} and the fact that, if $x_1 = \cos \t$ and $x_2 = \cos \phi$, then
$x_1+x_2 = \sqrt{1-\cos (\t-\phi)}\sqrt{1+\cos (\t+\phi)}$, we can deduce 
$$
|R_m^{\a,\b}(x_1,x_2)| \le \max_{-1 \le t\le 1} \left|\sqrt{1+t}\,p_m^{(\a,\b+1)}(t) \right |^2 \le c n^{2 \max\{\a,\b\} -1}
$$
from the well-known point-wise estimate of the Jacobi polynomials (cf. \cite[(4.3.4) and (7.32.5)]{Sz}). Using this
estimate and the Cauchy-Schwarz inequality, we conclude that
\begin{align*}
  \Lambda_m & \le  c n^{2 \max\{\a,\b\} -1} \bigg(\sum_{k=0}^n \sum_{j=0}^k \mu_{j,m}^{(\a,\b)}\mu_{k,m}^{(\a,\b)} 
      \left|R_m^{\a,\b}(s_{j,k}, t_{j,k}) \right| \bigg)^{\f12} \\
     & \le  c n^{2 \max\{\a,\b\} -1}  \sum_{k=0}^n \mu_{k,m}^{(\a,\b)}  (1+x_{k,m}) \left[p_{m}^{(\a,\b+1)}(x_{k,m})\right]^2. 
\end{align*}
Since $x_{k,m} = x_{k,m}^{(\a+1,b)}$, it follows from \eqref{eq:Jacobi3} and $p_{m}^{(\a,\b+1)}(x) = \CO(1) m^{\f12} 
P_{m}^{(\a,\b+1)}(x)$ (cf. \cite[(4.3.4)]{Sz}) that we can apply the Gauss-Radau rule of degree $2m$ to conclude
that 
$$
 \Lambda_m  \le  c n^{2 \max\{\a,\b\}} \int_{-1}^1 \left[P_{m-1}^{(\a+1,\b+1)}(t)\right]^2 w_{\a,\b+1}(t)dt 
     \le  c n^{2 \max\{\a,\b\}},
$$
where the last step follows from \cite[91, p. 391]{Sz}. This give the estimate of the sum over $Q_{m,m}^{\a,\b}$
part of $\CK_{2m+1}^* (x, \xb_{j,k}^{(1)})$ in the right hand side of \eqref{eq:Lebesgue}. 

Now, we claim that the sum over remaining parts of $\CK_{2m+1}^* (x, \xb_{j,k}^{(1)})$ can be deduced as in the proof of \cite{X12} for $n =2m$. In fact, it is easy to see that 
\begin{align*}
&(x_1+x_2)(y_1+y_2) K_m^{\a,\b+1,-\f12}(X,Y) \\
 &\quad = \frac12\left[k_m^{(\a,\b), 0,1}(\cos (\t_1-\t_2),\cos (\phi_1-\phi_2))
   k_m^{(\a,\b), 0,1}(\cos (\t_1+\t_2),\cos (\phi_1+\phi_2)) \right. \\
& \qquad \,\,  \left. + k_m^{(\a,\b), 0,1}(\cos (\t_1-\t_2),\cos (\phi_1+\phi_2))
   k_m^{(\a,\b), 0,1}(\cos (\t_1+\t_2),\cos (\phi_1-\phi_2))\right],
\end{align*}
where $k_m^{(\a,\b), 0,1}$ is defined at the bottom of \cite[p. 27]{X12}. Moreover, it is easy to verify that  
$\mu_{k,m}^{(\a,\b)} = \l_{k,m}^{(\a+1,\b)}/(1-x_{k,m}) \sim \l_{k,m}^{(\a,\b)}$ for $1 \le k \le m$, so that 
$\mu_{k,m}^{(\a,\b)}$ is comparable to $\l_{k,m}^{(\a,\b)}$ used in \cite{X12}. Since the main ingredient in
the proof of \cite{X12} is the pointwise estimate of the kernel $k_m^{(\a,\b), i,j}$, we can follow the proof there
to estimate the remaining sum over $1 \le k \le m$. The case $j=0$ of the sum can be easily handled, since 
$\mu_0 \sim n^{-\a-1}$ is small. This completes the proof. 
\end{proof}

In the case of $\a = \b = -1/2$, the order of the Lebesgue constant was determined in \cite{BMV} based 
on the explicit expression of the Lagrange interpolation polynomial given in \cite{X96}, which is different
from the one given here.

\end{document}